\documentclass[12pt]{amsart}

\newtheorem{theorem}{Theorem}[section]
\newtheorem{lemma}[theorem]{Lemma}

\newtheorem{proposition}[theorem]{Proposition}

\newtheorem{corollary}[theorem]{Corollary}
\newtheorem{definition}[theorem]{Definition}
\newtheorem{assumption}{Assumption}

\theoremstyle{remark}
\newtheorem{remark}[theorem]{Remark}

\newcommand{\klr}[1]{\left(#1\right)}
\newcommand{\kle}[1]{\left[#1\right]}
\newcommand{\kls}[1]{\left|#1\right|}

\newcommand{\kla}[1]{\langle#1\rangle}
\newcommand{\klg}[1]{\left\{#1\right\}}
\newcommand{\klgcases}[1]{\begin{cases}#1\end{cases}}
\newcommand{\norm}[1]{\left \lVert #1 \right \rVert}

\newcommand{\svec}[1]{\begin{pmatrix}#1\end{pmatrix}}
\newcommand{\matr}[1]{\svec{#1}}

\newcommand\nat[0]{\mathbb{N}}

\newcommand{\citewithin}[2]{\protect{\cite[#1]{#2}}}

\usepackage{comment}
\usepackage{tikz}
\usepackage{float}
\usepackage{relsize}
\usepackage{enumitem}
\usepackage[T1]{fontenc}
\usepackage[english]{babel}
\usepackage{color}
\usepackage{graphicx}
\usepackage[numbers]{natbib}
\usepackage{amsmath}
\usepackage{bbm}
\usepackage{mathtools}
\usepackage{csquotes}
\usepackage{todonotes}
\usepackage{newtxtext}
\usepackage[varvw]{newtxmath}

\setlist[enumerate]{label=\textnormal{(\roman*)}, font=\normalfont}

\usepackage[hidelinks, linktocpage=true, colorlinks=true,allcolors=blue,breaklinks=True]{hyperref}

\def\eqref#1{\textnormal{(\textcolor{blue}{\ref{#1}})}}
\renewcommand{\cite}[2][]{\textnormal{\citep[#1]{#2}}}
\newcommand{\condref}[2]{\textnormal{(\textcolor{blue}{\hyperref[#1]{#2}})}}

\date{\today}

\begin{document}

\numberwithin{equation}{section}

\title[Gaussian-type density estimates for mixed SDEs driven by correlated fBMs]{Gaussian-type density estimates for mixed SDEs driven by
correlated fractional Brownian motions} 

\author{Maximilian Buthenhoff}\thanks{}
\address{Maximilian Buthenhoff, Ruhr University Bochum, Institute for Theoretical
Physics III, Bochum, Germany}
\email{maximilian.buthenhoff\@@{}rub.de}

\author{Ercan S\"onmez}
\address{Ercan S\"onmez, Ruhr University Bochum, Faculty of Mathematics, Bochum, Germany}
\email{ercan.soenmez\@@{}rub.de}

\begin{abstract}
In this work, we investigate the existence and properties of Gaussian-like densities for weak solutions of multidimensional stochastic differential equations driven by a mixture of completely correlated fractional Brownian motions. We consider both the short-range and long-range dependent regimes, imposing a singular drift in the short-range dependent case and a Hölder continuous drift in the long-range dependent setting. Our approach avoids the use of Malliavin calculus and stochastic dynamical systems, relying instead on the Girsanov theorem and the framework of exponential Orlicz spaces. By considering a conditionally Gaussian process, we establish the existence of a density with respect to the Lebesgue measure. Furthermore, we derive Gaussian-type upper and lower bounds for this density, illustrating the optimality of our results in the short-range dependent case.
\end{abstract}

\maketitle

\allowdisplaybreaks

\section{Introduction}
\indent The goal of this paper is to show the existence of Gaussian-type bounds for the density of solutions of the multidimensional mixed stochastic differential equation (SDE)
\begin{align}
    X_t = x_0 + \int_0^t b(s,X_s) \mathrm{d}s +  A\klr{a_1 B_t^{H_1} + a_2 B_t^{H_2}}\,, \qquad t \in [0,T] \,,
    \label{eq:sde}
\end{align}
where we assume that the measurable drift function $b \colon [0,T] \times \mathbb{R}^d \to \mathbb{R}^d$ is bounded. Throughout this work, we fix $d\in \mathbb{N}$, $T\in (0,\infty)$, we let $x_0 \in \mathbb{R}^d$ be some initial condition, we choose $a_1,a_2 \in \mathbb{R}\setminus\klg{0}$ and we denote by $A \in \mathbb{R}^{d\times d}$ an invertible ($d\times d$)-matrix. Furthermore, the processes $(B^{H_1}_t)_{t \in [0,T]}$, $(B^{H_2}_t)_{t \in [0,T]}$ denote two completely correlated $d$-dimensional fractional Brownian motions with different Hurst parameters $H_1, H_2 \in (0,1)$, i.e.\ the integral representations can be expressed with respect to the same $d$-dimensional Brownian motion $(W_t)_{t \in [0,T]}$ as
\begin{align}
    B_t^{H_1} = \int_0^t K_{H_1}(t,s) \mathrm{d}W_s\,, \qquad 
    B_t^{H_2} = \int_0^t K_{H_2}(t,s) \mathrm{d}W_s\,,
    \label{eq:fbms-wrt-brownian}
\end{align}
for every $t\in [0,T]$, where 
\begin{align}
    K_H(t,s) = \Gamma(H + 1/2)^{-1} (t-s)^{H-1/2} F(H - 1/2, 1/2 - H, H + 1/2, 1-t/s) \,
    \label{eq:covariance-kernel}
\end{align}
represents the Volterra kernel of the fractional Brownian motion (fBM), see \cite{decreusefond1999stochastic}. The fBM $B^H$ with Hurst index $H \in (0,1)$ is the unique normalized zero-mean Gaussian stochastic process that has stationary increments and is self-similar of order $H$, i.e.\ for every $a > 0$
\begin{align*}
    B_{at}^H \overset{d}{=} a^H B_t^H\,, \quad t \ge 0.
\end{align*}
It represents a natural generalization of the Brownian motion which corresponds to fBM with Hurst index $H = 1/2$. \\  
\indent SDEs driven by fractional Brownian motion are used to model systems with memory and long-range dependence. One prominent example is finance where having a smooth density for the asset price or volatility is important for option pricing as it ensures well-defined likelihoods and enables the use of calibration techniques on observed price distributions \cite{gatheral2018volatility}. Fractional Gaussian noises are also popular for modeling anomalous diffusion \cite{kou2004generalized}, where the mean squared displacement of a particle grows non-linearly in time. In such physical systems, estimations of the density of the particle’s displacement are essential to predict the probability of finding the particle in a given region. \\
\indent In 2014, Baudoin et al.\ \cite{baudoin2014upper} extended the works of \cite{cass} and obtained the first Gaussian-type upper bound for densities of solutions of SDEs driven by a fBM with Hurst parameter $H \in (1/3,1)$. This result has been then generalized by Besalú et al.\ \cite{besalu2016gaussian} for a wider range of Hurst indices and further extended by Baudoin et al.\ \cite{baudoin2016probability} where a Gaussian-type lower bound has been obtained. However, all of these works considered smooth diffusion and drift functions only. A very recent breakthrough is the work of Li et al.\ \cite{li2023non} where they established Gaussian upper and lower bounds under minimal regularity requirements on the drift for $H \le 1/2$ and Hölder regularity for $H > 1/2$. Notably, their approach does not use Malliavin calculus; instead, their approach relies on stochastic dynamical systems developed by Hairer \cite{hairer2005ergodicity}. In this paper we present an alternative approach which does neither rely on Malliavin calculus nor on stochastic dynamical systems and is solely based on the Girsanov theorem and exponential Orlicz spaces. Moreover, unlike \cite{li2023non} we use the Volterra integral representation of fractional Brownian motion instead of the Mandelbrot–Van Ness representation. In the case of independent noise, our methods remain applicable with a simpler implementation strategy; moreover, they extend naturally to the setting of a linear combination of arbitrarily finitely many independent fractional Brownian motions.\\
\indent The strategy of the proof is the following. Based on \cite{nualart2021regularization}, we establish that, in arbitrary dimensions, the SDE~\eqref{eq:sde} possesses a weak solution if the drift function satisfies a linear growth condition in $x$, uniformly in $t$, for $\min(H_1,H_2) \le 1/2$, and a Hölder condition for $\min(H_1,H_2) > 1/2$. Moreover, any two weak solutions have the same probability law. Following the idea of~\cite{fan2021moment,sonmez2020mixed}, we introduce a conditionally Gaussian process (CGP) to show that the law of solutions of SDE~\eqref{eq:sde} admits a density with respect to the Lebesgue measure. Finally, by constructing similar auxiliary stochastic processes as in \cite{li2023non}, by using the Girsanov theorem and exponential Orlicz spaces, we conclude the existence of Gaussian-type upper and lower bounds. It is noteworthy to mention that the CGP used here represents a stochastic process that approximates the solution of the parameter-dependent two-fractional noise SDE and possesses the useful property that it is conditionally Gaussian. The CGP enables us to show not only the existence of a density with respect to the Lebesgue measure for the solutions of the SDE \eqref{eq:sde}, but also serves as a necessary tool in a parallel project \cite{buthenhoff} which deals with the existence and uniqueness of strong solutions for a related SDE. \\
\indent This work is organized as follows. In Section \ref{sec:prel}, we recall the basics of fractional calculus, define fractional Brownian motion and summarize its most important properties. Moreover, we introduce the Girsanov theorem and state our main result. Next, in Section \ref{sec:2f-sde}, we continue with the proof of the weak existence and uniqueness of solutions for the parameter-dependent two-fractional noise SDE. After that, in Section \ref{sec:exist-dens-cgp}, we introduce the CGP and show that the law of the solution of the mixed SDE admits a density with respect to the Lebesgue measure. Lastly, in Section \ref{sec:gaussian-bounds}, we complete the proof of the main result. 
\section{Preliminaries and main result}  \label{sec:prel}
\subsection{Fractional Brownian motion}
Let us denote by $(\Omega, (\mathcal{A}_t)_{t\in [0,T]},\mathbb{P})$ a complete filtered probability space. For the sake of completeness, we recall the definition of fractional Brownian motion and summarize properties we need to use throughout this work. For a detailed review of fBM, its properties and its applications in finance, we refer to the work by Mandelbrot and Van Ness \cite{mandelbrot1968fractional} and the book by Mishura \cite{mishura2008stochastic}. The (one-sided, normalized) fractional Brownian motion with Hurst index $H \in (0,1)$ is a Gaussian process $B^H = (B_t^H)_{t\in [0,T]}$ on $(\Omega, (\mathcal{A}_t)_{t\in [0,T]},\mathbb{P})$ with
\begin{enumerate}
    \item[(i)] $B_0^H = 0$,
    \item[(ii)] $\mathbb{E}[B_t^H] = 0$ for all $t \ge 0$,
    \item[(iii)] $\mathbb{E}[B_t^H B_s^H] = \frac{1}{2}\klr{|t|^{2H} + |s|^{2H} - |t-s|^{2H}} \mathbbm{1}_{d\times d}\eqqcolon R_H(t,s)$ for all $s,t \ge 0$,
\end{enumerate} 
where $\mathbbm{1}_{d\times d}$ denotes the ($d\times d$)-identity matrix. These properties imply that fBM is self-similar of order $H$. Moreover, combining the self-similarity with Kolmogorov's continuity criterion (cf.\ for example~\cite{stroock1997multidimensional}) it turns out that the fractional Brownian motion with Hurst index $H$ admits a version whose paths are almost surely Hölder continuous of order strictly less than $H$. \\
\indent As already mentioned in the introduction, according to \cite{decreusefond1999stochastic}, the fBM represents a Volterra process and possesses the integral representation 
\begin{align}
    B_t^H = \int_0^t K_H(t,s) \mathrm{d}W_s \,,
    \label{eq:integral-repr-fbm}
\end{align}
where $(W_t)_{t \in [0,T]}$ denotes a $d$-dimensional Brownian motion. From now on we will assume that $(\mathcal{A}_t)_{t\in [0,T]}$ is the augmented filtration generated by $(W_t)_{t \in [0,T]}$ satisfying the usual conditions. We call $K_H(t,s)$ the covariance kernel and it is given by equation \eqref{eq:covariance-kernel}. The covariance kernel $K_H(t,s)$ satisfies the following bounds (see \cite{decreusefond1999stochastic, nualart2006malliavin}):
\begin{enumerate}
    \item[(i)] For $H \in (0,1)$ we have
    \begin{align}
        |K_H(t,s)| \le C_1(H) s^{-|H-1/2|} (t-s)^{\max\klg{-1/2+H,0}} \mathbbm{1}_{[0,t]}(s)\,.
        \label{eq:cov-bound-1}
    \end{align}
    \item[(ii)] For $H \in (1/2,1)$ we have
    \begin{align}
        |K_H(t,s)| \ge C_2(H) (t-s)^{H - 1/2} \mathbbm{1}_{[0,t]}(s)\,.
        \label{eq:cov-bound-2}
    \end{align}
    \item[(iii)] For $H \in (0,1/2)$ we have
    \begin{align}
        |K_H(t,s)| \ge C_3(H) t^{H-1/2} s^{1/2-H} (t-s)^{H - 1/2} \mathbbm{1}_{[0,t]}(s)\,.
        \label{eq:cov-bound-3}
    \end{align}
\end{enumerate}
Here, $C_1(H),C_2(H)$ and $C_3(H)$ represent some positive constants that only depend on $H$. The covariance kernel of the fBM induces the existence of a covariance operator on the space of square-integrable functions which is defined as the map $K_H \colon L^2([0,T]) \to L^2([0,T])$ with
\begin{align*}
    (K_H h)(t) = \int_0^t K_H(t,s) h(s) \mathrm{d}s \,.
\end{align*}
Since the covariance kernel admits a representation in terms of the hypergeometric function (cf.\ equation \eqref{eq:covariance-kernel}), according to \cite[\S 10]{samko1993aa} we have a representation in terms of fractional operators, in particular, in terms of the fractional Riemann-Liouville integrals. \\
\indent Let $f \in L^1([a,b])$ and $\alpha > 0$. The left fractional Riemann-Liouville integral of $f$ of order $\alpha$ on $(a,b)$ is given at almost all $x$ by
\begin{align*}
   I_{a+}^\alpha f(x) = \frac{1}{\Gamma(\alpha)} \int_a^x (x-y)^{\alpha - 1} f(y) \mathrm{d}y \,.
\end{align*}
In the following, we denote by $I_{a+}^\alpha (L^p)$ the image of $L^p([a,b])$ by the operator $I_{a+}^\alpha$, i.e.\ the set of functions $f$ that can be represented as $f = I_{a+}^\alpha \varphi$ for some $\varphi \in L^p([a,b])$. Moreover, we define the left fractional Riemann Liouville derivative of $f \colon [a,b] \to \mathbb{R}$ of order $\alpha \in [0,1)$ by
\begin{align*}
   D^\alpha_{a+} f(x) = \frac{\mathrm{d}}{\mathrm{d}x} I^{1-\alpha}_{a+} f(x)
\end{align*}
if the right-hand side exists. Using these definitions, in \cite[\S 10]{samko1993aa} it has been shown that the covariance operator $K_H$ is an isomorphism from $L^2([0,T])$ onto $I_{0+}^{H+1/2}(L^2)$ and can be expressed as
\begin{align*}
    (K_H h)(s) = \klgcases{
        I^{2H}_{0+} s^{1/2-H} I^{1/2 - H}_{0+} s^{H - 1/2} h & \text{if } H < 1/2\\
        I^1_{0+} s^{H - 1/2} I_{0+}^{H-1/2} s^{1/2-H} h & \text{if } H > 1/2 \,.
    }
\end{align*}
Moreover, from the last equation we deduce that the inverse operator $K_H^{-1}$ is given by
\begin{align}
    (K_H^{-1}h)(s) = \klgcases{
        s^{1/2-H} D_{0+}^{1/2-H} s^{H-1/2} D_{0+}^{2H}h & \text{if } H < 1/2\\
        s^{H-1/2} D_{0+}^{H-1/2} s^{1/2-H} h' & \text{if } H > 1/2 \,.
    } 
    \label{remark:inverse-operator-representation}
\end{align}
In particular, if $h$ is absolutely continuous, in \cite[\S 3.1]{nualart2002regularization} it has been shown for $H < 1/2$ that
\begin{align}
    (K_H^{-1}h)(s) = s^{H-1/2} I_{0+}^{1/2-H} s^{1/2-H} h' \,.
    \label{eq:formula-inverse-H-smaller-12}
\end{align}
\subsection{Girsanov theorem} As a consequence of the multidimensional Girsanov theorem \cite{girsanov1960transforming}, we have the following variant of the Girsanov theorem for two fBMs that is an extension of the one in~\cite{nualart2021regularization}, with modifications to the multidimensional setting.
\begin{theorem}[Girsanov theorem for two fBMs]\label{thm:girsanov-two-fbm}
    Let $B^{H_1}, B^{H_2}$ be two completely correlated fBMs with Hurst parameters $H_1$ and $H_2$ as above. Moreover, let $(u_t)_{t \in [0,T]}$ and $(v_t)_{t \in [0,T]}$ be $(\mathcal{A}_t)_{t\in [0,T]}$-adapted processes with integrable paths such that $\int_0^{\cdot} u_s \mathrm{d}s \in I_{0+}^{H_1+1/2}(L^2)$ and $\int_0^{\cdot} u_s \mathrm{d}s \in I_{0+}^{H_2+1/2}(L^2)$ almost surely. Suppose that
    \begin{align}
        \psi_t \coloneqq \klr{K_{H_1}^{-1}\int_0^{\cdot} u_r \mathrm{d}r}(t) = \klr{K_{H_2}^{-1}\int_0^{\cdot} v_r \mathrm{d}r}(t)
        \label{eq:process-psi-two-fbms}
    \end{align}
    for every $t \in (0,T]$. If
    \begin{align*}
        L_T \coloneqq \exp\kle{\int_0^T \kla{\psi_s, \mathrm{d}W_s} - \frac{1}{2} \int_0^T |\psi_s|^2 \mathrm{d}s}
    \end{align*}
    satisfies $\mathbb{E}[L_T] = 1$, then, under the probability measure $\mathbb{Q}$ defined by $\mathrm{d}\mathbb{Q}/\mathrm{d}\mathbb{P} = L_T$, the processes $(\tilde{B}_t^{H_1})_{t \in [0,T]}$, $(\tilde{B}_t^{H_2})_{t \in [0,T]}$ with
    \begin{align*}
        \tilde{B}_t^{H_1} = B_t^{H_1} - \int_0^t u_s \mathrm{d}s\,, \qquad \tilde{B}_t^{H_2} = B_t^{H_2} - \int_0^t v_s \mathrm{d}s
    \end{align*}
    are completely correlated fractional Brownian motions. This means that there exists a $d$-dimensional Brownian motion $(\tilde{W}_t)_{t\in[0,T]}$ such that for all $t \in [0,T]$ we have
    \begin{align*}
        \tilde{B}_t^{H_1} = \int_0^t K_{H_1}(t,s) \mathrm{d}\tilde{W}_s\,, \qquad \tilde{B}_t^{H_2} = \int_0^t K_{H_2}(t,s) \mathrm{d}\tilde{W}_s\,.
    \end{align*}
\end{theorem}
Girsanov theorems are useful to construct probability measures with suitable properties. Note that if the process $(\psi_t)_{t\in [0,T]}$ defined in equation \eqref{eq:process-psi-two-fbms} fulfills the Novikov condition
\begin{align}
    \mathbb{E}\kle{\exp\klr{\frac{1}{2} \int_0^T |\psi_s|^2 \mathrm{d}s}} < \infty\,,
    \label{eq:novikov-condition}
\end{align}
then we have automatically $\mathbb{E}[L_T] = 1$.
\subsection{Main result}
After having introduced the necessary framework, we are now in position to state our main result. We impose the following assumption:
\begin{assumption}\label{assumpt:main-thm}
    Let $a_1,a_2 \in \mathbb{R}\setminus\klg{0}$ and assume that $A \in \mathbb{R}^{d\times d}$ is an invertible ($d\times d$)-matrix. The drift function $b$ is bounded in $x$ uniformly in $t$. Moreover, if $H \coloneqq \min(H_1,H_2) > 1/2$, additionally assume that the function $b$ is Hölder continuous, i.e., there exists a $c \in (0,\infty)$ such that for all $(t, x), (s, y) \in [0,T] \times \mathbb{R}^d$ we have
    \begin{align*}
        |b(t,x) - b(s,y)| \le c(|x-y|^\beta + |s-t|^\gamma)\,,
    \end{align*}
    where $\beta \in (1-1/(2H), 1)$ and $\gamma > H-1/2$.
\end{assumption}
\begin{theorem}\label{thm:main-thm}
    Assume that the drift in SDE \eqref{eq:sde} satisfies Assumption \ref{assumpt:main-thm} and let
    \begin{align*}
        \sigma^2(T_0) = \int_0^{T_0} \klr{ a_1 K_{H_1}(T_0,t) + a_2 K_{H_2}(T_0,t)}^2 \mathrm{d}t, \quad T_0 \in (0,T]  \,.
    \end{align*}
    For each $T_0 \in (0,T]$ the law of the solution $X_{T_0}$ of the SDE \eqref{eq:sde} admits a Lebesgue density $p_{T_0}$ and there exist $C_1,C_1' = \mathrm{const}(H_1,H_2,A,a_1,a_2,d,T_0) > 0$, $C_2,C_2' = \mathrm{const}(H_1,H_2,A,a_1,a_2,d) > 0$, such that this density admits the following Gaussian bounds for all $x \in \mathbb{R}^d$:
    \begin{align*}
        \frac{C_1'}{\sigma^d(T_0)} \exp\klr{-\frac{C'_2|x-x_0|^2}{\sigma^2(T_0)}}\, \leq p_{T_0}(x) &\le \frac{C_1}{{\sigma^d(T_0)}} \exp\klr{-\frac{C_2|x-x_0|^2}{\sigma^2(T_0)}}\,.
    \end{align*} 
    Furthermore, if $\min(H_1,H_2) \leq 1/2$ the constants $C_1,C_1'$ are independent of $T_0$ as well.
\end{theorem}
The exact definitions of the constants $C_1,C_1', C_2,C_2'$ are deferred to Section \ref{sec:main-result-proof}. In the proofs we will see that their behavior strongly depends on whether the Hurst indices lie in the short- or long-range dependent regime. Moreover, it follows from our analysis that the bounds in Theorem~\ref{thm:main-thm} are optimal in the case $\min(H_1,H_2) \leq 1/2$, where in particular the constants $C_1, C_1'$ are independent of $T_0$. If $\min(H_1,H_2) > 1/2$, the behavior of $C_1, C_1'$ in $T_0$ will be made explicit in the proofs. This behavior aligns with similar observations as in \cite{baudoin2014upper}.

%
%
\section{Weak existence and uniqueness of solutions}\label{sec:2f-sde}
In this section, for the sake of completeness, we give a proof for the weak existence and uniqueness of solutions of SDE \eqref{eq:sde} under the following assumption on the drift function $b \colon [0,T] \times \mathbb{R}^d \to \mathbb{R}^d$. The techniques used in proving these results will be essential for the proof of our main result, as they provide the foundational tools required for the derivation of the density bounds established in Theorem~\ref{thm:main-thm}.
\begin{assumption}\label{assumpt:mixed-sde-solution}
    Let $a_1,a_2 \in \mathbb{R}\setminus\klg{0}$, assume that $A \in \mathbb{R}^{d\times d}$ is an invertible ($d\times d$)-matrix and define $H \coloneqq \min(H_1,H_2)$. Let one of the following two conditions be fulfilled:
    \begin{enumerate}
        \item $H \le 1/2$ and the function $b$ satisfies a linear growth condition in $x$ uniformly in $t$, i.e., we assume that there exists a $c \in (0,\infty)$ such that for all $(t,x) \in [0,T] \times \mathbb{R}^d$ we have
        \begin{align}
            |b(t,x)| \le c(1 + |x|) \,.
            \label{eq:linear-growth-condition}
        \end{align}
        \item $H > 1/2$ and the function $b$ satisfies a Hölder condition, i.e., there exists a $c \in (0,\infty)$ such that for all $(t, x), (s, y) \in [0,T] \times \mathbb{R}^d$ we have
        \begin{align}
            |b(t,x) - b(s,y)| \le c(|x-y|^\beta + |s-t|^\gamma)\,,
            \label{eq:holder-condition}
        \end{align}
        where $\beta \in (1-1/(2H),1)$ and $\gamma > H-1/2$.
    \end{enumerate}
\end{assumption}
%

Note that Assumption \ref{assumpt:main-thm} implies Assumption \ref{assumpt:mixed-sde-solution}. In our proofs we will rely on Assumption \ref{assumpt:main-thm}, for the well-posedness of the equation we work with the more general Assumption \ref{assumpt:mixed-sde-solution}.

Recall the following definition. 
\begin{definition}[Weak solution]
    A weak solution to the stochastic differential equation \eqref{eq:sde} is defined as a triple $(B^{H_1}, B^{H_2}, X)$ on a filtered probability space $(\Omega, (\mathcal{A}_t)_{t \in [0,T]}, \mathbb{P})$ such that the following conditions hold
    \begin{enumerate}
        \item There exists an $(\mathcal{A}_t)_{t \in [0,T]}$-Brownian motion $W = (W_t)_{t \in [0,T]}$ such that the integral representations of the fBMs satisfy equation \eqref{eq:fbms-wrt-brownian}.
        \item The triple $(B^{H_1}, B^{H_2}, X)$ satisfies the stochastic differential equation \eqref{eq:sde} and
        \begin{align*}
            \int_0^T |b(s,X_s)| \mathrm{d}s < \infty
        \end{align*}
        almost surely. 
    \end{enumerate}
\end{definition}
%
%
\subsection{Existence of weak solution}
The strategy we follow is based on the works \cite{nualart2002regularization,nualart2021regularization,nualart2022regularization}.
\begin{proposition}[cf.\ \citewithin{Theorem 3.2}{nualart2021regularization}]\label{prop:helper-weak-existence}
    Define $H \coloneqq \min(H_1,H_2)$ and let $h \colon [0,T] \to \mathbb{R}^d$ be a measurable function that satisfies one of the following conditions:
    \begin{enumerate}
        \item $H \le 1/2$ and the function $h$ satisfies $\norm{h}_\infty < \infty$.
        \item $H > 1/2$ and the function $h$ is $\gamma$-Hölder continuous, i.e., there exists a $c \in (0,\infty)$ such that for all $s,t \in [0,T]$ we have
        \begin{align}
            |h(t) - h(s)| \le c |t-s|^\gamma\,,
            \label{eq:h-holder}
        \end{align}
        where $\gamma > H - 1/2$. 
    \end{enumerate}
    Then, there exist $(\mathcal{A}_t)_{t \in [0,T]}$-adapted processes $(u_t)_{t \in [0,T]}$, $(v_t)_{t\in [0,T]}$ with integrable paths such that
    \begin{align*}
        \int_0^{\cdot} u_s \mathrm{d}s \in I_{0+}^{H_1+1/2}(L^2)\,, \qquad 
        \int_0^{\cdot} v_s \mathrm{d}s \in I_{0+}^{H_2 + 1/2}(L^2)
    \end{align*}
    almost surely and the following conditions are satisfied:
    \begin{enumerate}
        \item[(A1)] \label{cond:A1}For all $t \in [0,T]$ we have
        \begin{align*}
            A (a_1u_t + a_2v_t )= h(t) \,.
        \end{align*}
        \item[(A2)] \label{cond:A2} For all $t \in [0,T]$ we have
        \begin{align*}
             \psi_t \coloneqq \klr{K_{H_1}^{-1}\int_0^{\cdot} u_r \mathrm{d}r}(t) = \klr{K_{H_2}^{-1}\int_0^{\cdot} v_r \mathrm{d}r}(t)
        \end{align*}
        and the process $(\psi_t)_{t \in [0,T]}$ satisfies the Novikov condition \eqref{eq:novikov-condition}.
        \item[(A3)] \label{cond:A3} For all $t \in [0,T]$ we have
        \begin{align*}
            |\psi_t| \le C(H_1,H_2,\norm{A},T) |a|^{-1} \klgcases{
                \norm{h}_\infty t^{1/2-H} & \text{for } H \le 1/2\\
                \klr{|h(0)| + \norm{h}_\gamma T^\gamma + G} t^{1/2 - H} & \text{for } H > 1/2
            } \,,
        \end{align*}
        where $C(H_1,H_2,\norm{A},T)$ is some $H_1$, $H_2$, $\norm{A}$ and $T$ dependent constant and $G$ is a random variable that satisfies
        \begin{align*}
            \mathbb{E}[\exp(\lambda G^2)] < \infty
        \end{align*}
        for all $\lambda > 0$. Moreover, we set $a = a_1$ if $H_1 \le H_2$ and $a = a_2$ otherwise. Here, $\norm{A} = \max_{|y| =1 }|Ay|$ denotes the usual matrix norm. Note that $t^{1/2-H} \le T^{1/2-H}$ for $H \le 1/2$.
    \end{enumerate}
\end{proposition}
\begin{proof}
    We follow the proof of \cite[Theorem 3.2]{nualart2021regularization}. In total we need to distinguish between four cases. Recall that for any $\beta > -1$ and $\alpha \in (0,1)$ \cite[Lemma 3.2]{nualart2022regularization}
    \begin{align}
        I_{0+}^\alpha t^\beta = \frac{\Gamma(\beta + 1)}{\Gamma(\alpha + \beta + 1)}t^{\alpha+\beta} 
        \label{eq:helper-prop-3.3}
    \end{align}
    and recall the Mittag-Leffler function \cite{gorenflo2020mittag,haubold2011mittag} defined by
    \begin{align}
        E_{a,b}(x) \coloneqq \sum_{n=0}^\infty \frac{x^n}{\Gamma(an + b)}
        \label{eq:mittag-leffler}
    \end{align}
    for $a,b > 0$, whereby it satisfies the following exponential bound \cite[\S 4.4]{gorenflo2020mittag}:
    \begin{align*}
        |E_{a,b}(x)| \le M_1 \exp\left(M_2 |x|^{1/a}\right) 
    \end{align*}
    for some $M_1 > 0$ and any $M_2 > 1$. In the following, $C(H_1,H_2),C(H_1,H_2,\norm{A}) > 0$ denote constants that change throughout this proof. However, the actual value is not important for the validity of the statement.
    \paragraph{\textit{Case 1: $H_1 = 1/2$.}} Suppose that $H_2 < 1/2$. Define $\alpha = 1/2 - H_2 > 0$, set
    \begin{align*}
        Aa_2v_t = t^{-\alpha} \kle{1 + I_{0+}^\alpha}^{-1} \klr{t^\alpha h(t)}
    \end{align*}
    and let
    \begin{align*}
        Aa_1 u_t = a_2 t^{-\alpha} I_{0+}^\alpha t^\alpha v_t 
    \end{align*}
    for all $t \in (0,T]$. Then, $A(a_1 u_t + a_2v_t) = h(t)$ as required by condition \condref{cond:A1}{A1}. Moreover, due to equation~\eqref{remark:inverse-operator-representation} condition \condref{cond:A2}{A2} is also satisfied. By using equation \eqref{eq:helper-prop-3.3}, we obtain
    \begin{align*}
        |\psi_t| = |u_t| &\le C(H_1,H_2,\norm{A}) |a_2|/|a_1| T^{1/2-H_2} |v_t| \,.
    \end{align*}
    Since
    \begin{align*}
        \kle{1 + I_{0+}^\alpha}^{-1}(t^\alpha h(t)) = \sum_{n=0}^\infty (-1)^n I_{0+}^{n\alpha}(t^\alpha h(t))\,,
    \end{align*}
    using the triangle inequality and equation \eqref{eq:helper-prop-3.3} again we can estimate $|v_t|$ as
    \begin{align*}
        |v_t| 
        &\le C(H_1,H_2,\norm{A}) |a_2|^{-1} \norm{h}_\infty E_{\alpha,\alpha+1}(T^\alpha)
    \end{align*}
    where $E_{a,b}(x)$ represents the Mittag-Leffler function \eqref{eq:mittag-leffler}, i.e., condition \condref{cond:A3}{A3} is also satisfied. \\
    \indent Now suppose that $H_2 > 1/2$ and define $\alpha = H_2-1/2 > 0$. Let
    \begin{align*}
        Aa_1 u_t = t^\alpha \sum_{n=1}^\infty \frac{(-1)^n}{\Gamma(n\alpha)} \int_0^t (t-s)^{k\alpha - 1} s^{-\alpha} h(s) \mathrm{d}s + h(t)
    \end{align*}
    and 
    \begin{align*}
        Aa_2 v_t = a_1 t^\alpha \frac{1}{\Gamma(\alpha)} \int_0^t (t-s)^{\alpha-1} s^{-\alpha} u_s \mathrm{d}s
    \end{align*}
    for every $t \in (0,T]$. Now we can argue as in the case $H_2 < 1/2$ and conclude the validity of conditions \condref{cond:A1}{A1}, \condref{cond:A2}{A2} and \condref{cond:A3}{A3}.
    \\ \indent From now on, without loss of generality, assume that $H_1 \le H_2$.
    \paragraph{\textit{Case 2: $H_1,H_2 \in (0,1/2)$.}} Define $\alpha_1 = 1/2 - H_1$ and $\alpha_2 = 1/2 - H_2$. Consider
    \begin{align*}
        Aa_1 u_t = t^{-\alpha_2} \kle{1 + D_{0+}^{\alpha_2} t^{\alpha_2 - \alpha_1} I_{0+}^{\alpha_1} t^{\alpha_1 - \alpha_2}}^{-1}(t^{\alpha_2} h(t))
    \end{align*}
    and
    \begin{align*}
        Aa_2 v_t = a_1 t^{-\alpha_2} D_{0+}^{\alpha_2} t^{\alpha_2-\alpha_1} I_{0+}^{\alpha_1} t^{\alpha_1} u_t \,.
    \end{align*}
    Clearly, $A(a_1 u_t + a_2 v_t )= h(t)$ and by definition condition \condref{cond:A2}{A2} is also fulfilled. Now, take a look at 
    \begin{align*}
        \psi_t = \klr{K_{H_1}^{-1} \int_0^{\cdot} u_s \mathrm{d}s}(t) = t^{-\alpha_1} I_{0+}^{\alpha_1} t^{\alpha_1}(u_t) \,.
    \end{align*}
    According to equation \eqref{eq:helper-prop-3.3}, we can estimate $\psi_t$ as
    \begin{align*}
        |\psi_t| &\le C(H_1,H_2,\norm{A}) T^{1/2-H_1} |u_t| \,.
    \end{align*}
    Now we can follow the same steps as in the proof of \cite[Theorem 3.2]{nualart2021regularization} to estimate $|u_t|$ and conclude that
    \begin{align*}
        |\psi_t| \le C(H_1,H_2,\norm{A}) |a_1|^{-1} T^{1/2 - H_1} \norm{h}_\infty 
    \end{align*}
    as needed.
    \paragraph{\textit{Case 3: $H_1 \in (0,1/2)$, $H_2 \in (1/2,1)$.}} Let $\alpha_1 = 1/2 - H_1$, $\alpha_2 = H_2 - 1/2$ and define the processes $u,v$ by
    \begin{align*}
        Aa_1u_t &= h(t) + t^{-\alpha_1} \sum_{n=1}^\infty (-1)^n \kle{t^{\alpha_1 + \alpha_2} \klr{I_{0+}^{\alpha_2} \kle{t^{-\alpha_1-\alpha_2} I_{0+}^{\alpha_1}}}}^n (t^{\alpha_1} h(t)) \\
        &= t^{-\alpha_1}\kle{1 + t^{\alpha_1 + \alpha_2} I_{0+}^{\alpha_2}\kle{t^{-\alpha_1 - \alpha_2}I_{0+}^{\alpha_1}}}^{-1}(t^{\alpha_1}h(t))
    \end{align*}
    and
    \begin{align*}
        Aa_2v_t &= a_1 t^{\alpha_2} I_{0+}^{\alpha_2}\kle{t^{-\alpha_1-\alpha_2}I_{0+}^{\alpha_1}(t^{\alpha_1}u_t)} \,.
    \end{align*}
    Then it follows that $A(a_1u_t + a_2v_t )= h(t)$, i.e., condition \condref{cond:A1}{A1} is satisfied. Furthermore, because of equation \eqref{remark:inverse-operator-representation} condition \condref{cond:A2}{A2} is also satisfied. Now we need to bound the process $\psi$. Since $h \colon [0,T] \to \mathbb{R}$ represents a measurable function that satisfies $\norm{h}_\infty \coloneqq \sup_{t \in [0,T]} |h(t)| < \infty$, as a direct consequence of \cite[Lemma 3.3]{nualart2021regularization}, the process
    \begin{align*}
        Aa_1 u_t \coloneqq t^{-\alpha_1}\kle{1 + t^{\alpha_1 + \alpha_2} I_{0+}^{\alpha_2}\kle{t^{-\alpha_1 - \alpha_2}I_{0+}^{\alpha_1}}}^{-1}(t^{\alpha_1}h(t))
    \end{align*}
    satisfies
    \begin{align*}
        |a_t| |Au_t| \le C_1(H_1,H_2) \norm{h}_\infty \sum_{n=0}^\infty \frac{t^{n(\alpha_1 + \alpha_2)}}{\Gamma((n+1)\alpha_1 + (n-1)\alpha_2 + 1)} \,.
    \end{align*}
    Identifying the Mittag-Leffler function \eqref{eq:mittag-leffler} we further estimate
    \begin{align*}
        |a_1| |Au_t| &\le C_1(H_1,H_2) \norm{h}_\infty |E_{H_2 - H_1, 2-H_1-H_2}(T^{H_2 - H_1})| \,.
    \end{align*}
    Using equation \eqref{eq:formula-inverse-H-smaller-12} we obtain
    \begin{align*}
        \psi_t = t^{H_1 - 1/2} I_{0+}^{1/2-H_1} t^{1/2-H_1}u_t \,.
    \end{align*}
    Now using equation \eqref{eq:helper-prop-3.3} we get the following bound
    \begin{align*}
        |\psi_t| \le C(H_1,H_2,\norm{A}) |a_1|^{-1} T^{1/2 - H_1} E_{H_2 - H_1, 2-H_1-H_2}(T^{H_2 - H_1}) \norm{h}_\infty 
    \end{align*}
    for all $a_1 \neq 0$ where $C$ represents a constant that only depends on $H_1$ and $H_2$, i.e., condition \condref{cond:A3}{A3} is also satisfied.
    \paragraph{\textit{Case 4: $H_1,H_2 \in (1/2,1)$.}} Let $\alpha_1 = H_1 - 1/2 > 0$ and $\alpha_2 = H_2 - 1/2 > 0$. Define
    \begin{align*}
        Aa_1u_t &= t^{\alpha_2} \kle{1 + I_{0+}^{\alpha_2} t^{\alpha_1-\alpha_2} D_{0+}^{\alpha_1} t^{\alpha_2-\alpha_1}}^{-1} (t^{-\alpha_2} h(t)) \\
        &= h(t) + t^{\alpha_2} \sum_{n=1}^\infty (-1)^n \klr{I_{0+}^{\alpha_2} t^{\alpha_1-\alpha_2} D_{0+}^{\alpha_1} t^{\alpha_2-\alpha_1}}^n (t^{-\alpha_2} h(t))
    \end{align*}
    and
    \begin{align*}
        Aa_2v_t = a_1 t^{\alpha_2} I_{0+}^{\alpha_2} t^{\alpha_1-\alpha_2} D_{0+}^{\alpha_1} t^{-\alpha_1} u_t
    \end{align*}
    for every $t \in (0,T]$. Then \condref{cond:A1}{A1} and \condref{cond:A2}{A2} are satisfied. Using equation \eqref{remark:inverse-operator-representation} we can estimate

    \begin{align*}
        |\psi_t| &\le |t^{\alpha_1} D_{0+}^{\alpha_1} t^{-\alpha_1} u_t| \\
        &\le \norm{A}^{-1}|a_1|^{-1} \klr{ |t^{\alpha_1} D_{0+}^{\alpha_1} t^{-\alpha_1} h(t)| + \sum_{n=1}^\infty \kls{t^{\alpha_1} D_{0+}^{\alpha_1} t^{\alpha_2-\alpha_1} \klr{I_{0+}^{\alpha_2} t^{\alpha_1-\alpha_2} D_{0+}^{\alpha_1} t^{\alpha_2-\alpha_1}}^n (t^{-\alpha_2} h(t))}}   \,.
    \end{align*}
    According to \cite[Lemma 3.1]{nualart2022regularization}, for any $\beta > -1$, $\alpha \in (0,1)$ and any $g \in I_{0+}^\alpha([0,T])$ we have
    \begin{align*}
        D_{0+}^\alpha t^\beta g_t = t^\beta D_{0+}^\alpha g_t + \frac{\alpha}{\Gamma(1-\alpha)} \int_0^t g_s \frac{t^\beta - s^\beta}{(t-s)^{\alpha+1}} \mathrm{d}s \,.
    \end{align*}
    Using this formula, the Hölder continuity \eqref{eq:h-holder} of $h$ and the triangle inequality, we can estimate the first term as
    \begin{align*}
        |t^{\alpha_1} D_{0+}^{\alpha_1} t^{-\alpha_1} h(t)| \le C(H_1,H_2) \klr{|h(0)| + \norm{h}_\gamma T^\gamma} t^{1/2-H_1} .
    \end{align*}
    To estimate the second term, we follow the proof of \cite[Theorem 3.2]{nualart2021regularization} to find
    \begin{align*}
        \sum_{n=1}^\infty \kls{t^{\alpha_1} D_{0+}^{\alpha_1} t^{\alpha_2-\alpha_1} \klr{I_{0+}^{\alpha_2} t^{\alpha_1-\alpha_2} D_{0+}^{\alpha_1} t^{\alpha_2-\alpha_1}}^n (t^{-\alpha_2} h(t))} \le G t^{1/2-H_1}\,,
    \end{align*}
    where $G$ is some random variable that satisfies
    \begin{align*}
        \mathbb{E}[\exp(\lambda |G|^2)] < \infty
    \end{align*}
    for all $\lambda > 0$. Combining this estimation with the estimation of the first term, we conclude that condition \condref{cond:A3}{A3} is also satisfied.
\end{proof}
The latter proposition enables us to show the existence of a weak solution to the stochastic differential equation \eqref{eq:sde}.
\begin{theorem}[Existence of weak solution]\label{thm:existence-weak-solution}
     Under Assumption \ref{assumpt:mixed-sde-solution}, there exists a weak solution to the stochastic differential equation \eqref{eq:sde}.
\end{theorem}
\begin{proof}
    \textit{Step 1}: Define the function $b \colon [0,T] \to \mathbb{R}^d$ by
    \begin{align*}
        b_t \coloneqq b(t,x_0 + A(a_1 B_t^{H_1} + a_2 B_t^{H_2})) \,.
    \end{align*}
    Now we would like to show the existence of a process $\psi$ which satisfies the conditions \condref{cond:A1}{A1}, \condref{cond:A2}{A2}, \condref{cond:A3}{A3} for $h(t) = b_t$ and the following condition:
    \begin{enumerate}
        \item[(A4)] \label{cond:A4} One can find a $\lambda > 0$ with
        \begin{align*}
            \sup_{[0,T]} \mathbb{E}[\exp(\lambda |\psi_t|^2)] < \infty \,.
        \end{align*}
    \end{enumerate}
    This is a consequence of Proposition \ref{prop:helper-weak-existence}. If $H \le 1/2$ and $b$ satisfies the linear growth condition~\eqref{eq:linear-growth-condition} there exists a $c \in (0,\infty)$ such that the process $(b(t,x_0 + A(a_1B_t^{H_1} + a_2B_t^{H_2})))_{t \in [0,T]}$ satisfies
    \begin{align*}
        \sup_{t \in [0,T]} |b_t| \le c \klr{1 + |x_0| + \norm{A} |a_1| \norm{B^{H_1}}_\infty + \norm{A} |a_2| \norm{B^{H_2}}_\infty} \,.
    \end{align*}
    Therefore, there exists a $\psi$ that satisfies condition \condref{cond:A1}{A1}, \condref{cond:A2}{A2} and \condref{cond:A3}{A3} and the inequality
    \begin{align*}
        |\psi_t|^2 &\le C^2 c^2 \klr{1 + |x_0|^2 + \norm{B^{H_1}}_\infty^2 + \norm{B^{H_2}}_\infty^2} \,.
    \end{align*}
    Now we can apply Fernique's theorem \cite{fernique1975regularite} to conclude that condition \condref{cond:A4}{A4} is also satisfied. \\
    \indent On the other hand, if $H > 1/2$ and $b$ satisfies the Hölder condition \eqref{eq:holder-condition}, we have
    \begin{align*}
        |b_t - b_s| \le c |t-s|^\gamma \klr{1 + \norm{A}|a_1| \norm{B^{H_1}}_{\gamma/\beta}^\gamma + \norm{A} |a_2| \klr{1 + \norm{B^{H_2}}^\gamma_{\gamma/\beta}}} \,.
    \end{align*}
    Therefore, due to Proposition \ref{prop:helper-weak-existence}, the conditions \condref{cond:A1}{A1}, \condref{cond:A2}{A2} and \condref{cond:A3}{A3} are also satisfied. Additionally, as $H - 1/2 < 1/2$, condition \condref{cond:A4}{A4} is also fulfilled.
    \\\\
    \textit{Step 2}: Since the process $\psi$ satisfies the condition \condref{cond:A4}{A4}, according to Friedmann \cite{friedman1975stochastic} this is sufficient for the validity of the Novikov criterion \eqref{eq:novikov-condition}. Therefore we can apply the Girsanov theorem for two fBMs (Theorem \ref{thm:girsanov-two-fbm}) to conclude that there exists a probability measure $\mathbb{Q}$ such that the processes
    \begin{align*}
        \tilde{B}_t^{H_1} = B_t^{H_1} - \int_0^t u_s \mathrm{d}s\,, \qquad \tilde{B}_t^{H_2} = B_t^{H_2} - \int_0^t v_s \mathrm{d}s, \quad t\in [0,T],
    \end{align*}
    are $(\mathcal{A}_t)_{t\in [0,T]}$-fractional Brownian motions with Hurst indices $H_1$ and $H_2$ (with respect to the measure $\mathbb{Q}$) satisfying
    \begin{align*}
        \Tilde{B}_t^{H_1} = \int_0^t K_{H_1}(t,s) \mathrm{d}\tilde{W}_s\,, \qquad 
        \Tilde{B}_t^{H_2} = \int_0^t K_{H_2}(t,s) \mathrm{d}\tilde{W}_s\,,\quad t\in [0,T],
    \end{align*}
    where $(\tilde{W}_t)_{t\in[0,T]}$ is an $(\mathcal{A}_t)_{t\in [0,T]}$-Brownian motion. Now we can set
    \begin{align*}
        X_t = x_0 + A(a_1B_t^{H_1} + a_2B_t^{H_2}) \,.
    \end{align*}
    Then the triple $(\tilde{B}^{H_1}, \tilde{B}^{H_2}, X)$ is a weak solution to the SDE \eqref{eq:sde} on the filtered probability space $(\Omega, (\mathcal{A}_t)_{t \in [0,T]},\mathbb{Q})$, since for every $t \in [0,T]$ we have
    \begin{align*}
        A(a_1B_t^{H_1} + a_2B_t^{H_2}) &= A\klr{\int_0^t a_1K_{H_1}(t,s) \mathrm{d}W_s + \int_0^t a_2K_{H_2}(t,s)\mathrm{d}W_s} \\
        &= A\int_0^t a_1K_{H_1}(t,s) \klr{\mathrm{d}\tilde{W}_s + \psi_s \mathrm{d}s} + A\int_0^t a_2K_{H_2}(t,s) \klr{\mathrm{d}\tilde{W}_s + \psi_s \mathrm{d}s} \\
        &= A(a_1\tilde{B}_t^{H_1} + a_2\tilde{B}_t^{H_2} )+ \int_0^t b(s, x_0 + A(a_1B_s^{H_1} + a_2B_s^{H_2})) \mathrm{d}s\,,
    \end{align*}
    and, therefore
    \begin{align*}
        X_t = x_0 + A(a_1\tilde{B}_t^{H_1} + a_2\tilde{B}_t^{H_2} ) + \int_0^t b(s, X_s) \mathrm{d}s \,.
    \end{align*}
\end{proof}
\subsection{Uniqueness in law}
In the last paragraph we have shown that there exists a weak solution of the SDE \eqref{eq:sde}. Now we would like to show that the solution is weakly unique. Uniqueness in law is a consequence of Gronwall's lemma \cite{gronwall1919note}, Proposition \ref{prop:helper-weak-existence} and Girsanov's theorem \ref{thm:girsanov-two-fbm}.
\begin{theorem}\label{thm:uniqueness-law}
    Under Assumption \ref{assumpt:mixed-sde-solution}, two weak solutions of the SDE \eqref{eq:sde} have the same distribution.
\end{theorem}
\begin{proof}
    The following proof is based on \cite{nualart2002regularization,nualart2022regularization}. \textit{Step 1:} If $H \le 1/2$, due to the linear growth condition \eqref{eq:linear-growth-condition} we have
    \begin{align*}
        |b(t,x)| \le c \klr{1 + |x|} \,.
    \end{align*}
    Therefore, a weak solution $(B^{H_1}, B^{H_2}, X)$ on a filtered probability space $(\Omega, (\mathcal{A}_t)_{t \in [0,T]},\mathbb{P})$ satisfies
    \begin{align*}
        |X_t| &\le |x_0| +\norm{A}|a_1||B_t^{H_1}| + \norm{A}|a_2||B_t^{H_2}| + ct + c\int_0^t |X_s| \mathrm{d}s \\
        &\le \norm{A}|a_1|\norm{B^{H_1}}_\infty + \norm{A}|a_2|\norm{B^{H_2}}_\infty + c_1T + c\int_0^t |X_s| \mathrm{d}s \,,
    \end{align*}
    where $c_1 > 0$ is some constant. Now we can apply Gronwall's lemma \cite{gronwall1919note} for $|X_t|$ to conclude for all $t \in [0,T]$
    \begin{align*}
        |X_t| &\le \norm{A}\klr{|a_1|\norm{B^{H_1}}_\infty + \norm{A}|a_2|\norm{B^{H_2}}_\infty + c_1T} e^{cT} \,.
    \end{align*}
    On the other hand, if $H > 1/2$, due to the Hölder continuity we have
    \begin{align*}
        |X_t - X_s| \le \norm{A}|a_1| |B_t^{H_1} - B_s^{H_1}| + \norm{A}|a_2| |B_t^{H_2} - B_s^{H_2}| + c_2|t-s|(1 + \norm{X}_\infty)
    \end{align*}
    for some constant $c_2 > 0$. Now we use Proposition \ref{prop:helper-weak-existence} for the function $h(t) = b(t,X_t)$ and we can apply the Girsanov theorem to conclude that the process $(\tilde{W}_t)_{t \in [0,T]}$ with
    \begin{align*}
        \tilde{W}_t = W_t + \int_0^t \psi_s \mathrm{d}s\,, \qquad t \in [0,T]\,,
    \end{align*}
    where $\psi_t$ is defined as in the last paragraph, is a $d$-dimensional Brownian motion with respect to the measure $\mathbb{Q}$ defined by
    \begin{align*}
        \frac{\mathrm{d}\mathbb{Q}}{\mathrm{d}\mathbb{P}} = \exp\klr{- \int_0^T \kla{\psi_s, \mathrm{d}W_s} - \frac{1}{2} \int_0^T |\psi_s|^2 \mathrm{d}s} \,.
    \end{align*}
    Now we can use the SDE \eqref{eq:sde} and \condref{cond:A2}{A2} to express $X_t$ as
    \begin{align*}
        X_t &= x_0 + \int_0^t b(s,X_s) \mathrm{d}s + A(a_1 B_t^{H_1} + a_2 B_t^{H_2}) \\
        &= x_0 + Aa_1 (K_{H_1} \psi)(t) + Aa_2 (K_{H_2} \psi)(t) + Aa_1 B_t^{H_1} + Aa_2 B_t^{H_2} \\
        &= x_0 + Aa_1 \int_0^t K_{H_1}(t,s) \mathrm{d}\tilde{W}_s + Aa_2 \int_0^t K_{H_2}(t,s) \mathrm{d}\tilde{W}_s \,.
    \end{align*}
    Therefore for every bounded measurable functional $\Psi$ on $C([0,T])$ we have
    \begin{align*}
        \mathbb{E}_{\mathbb{P}}[\Psi(X)] = \mathbb{E}_{\mathbb{P}}\kle{\Psi\klr{x_0 + Aa_1 \int_0^t K_{H_1}(t,s) \mathrm{d}\tilde{W}_s + Aa_2 \int_0^t K_{H_2}(t,s) \mathrm{d}\tilde{W}_s}} \,.
    \end{align*}
    This equality holds for every weak solution $X$, i.e.\ we have uniqueness in law.
\end{proof}

\section{Existence of a density with respect to the Lebesgue measure}\label{sec:exist-dens-cgp}
In this section we would like to summarize some properties of the weak solution of the SDE~\eqref{eq:sde}, in particular, we define a conditionally Gaussian process and show that the solution admits a density with respect to the Lebesgue measure if we additionally assume that the drift function $b$ of the SDE is bounded. First of all, note that the solution of the SDE \eqref{eq:sde} fulfills the following property. In the following, we denote by $C_{\mathrm{H}}(Z)$ the Hölder constant of a process $Z$.
\begin{proposition}[Hölder continuity]\label{thm:holder-continuity-XN}
    Under Assumption \ref{assumpt:main-thm}, the solution $X$ of equation \eqref{eq:sde} admits a version whose paths are almost sure Hölder continuous of order $\gamma \in (0, H_1)$ with Hölder constant
    \begin{align*}
        C_{\mathrm{H}}(X) \le \norm{A} \mathrm{max}\klg{|a_1| C_{\mathrm{H}}(B^{H_1}), |a_2| C_{\mathrm{H}}(B^{H_2})}\,.
    \end{align*}
    Moreover, all moments of the Hölder constant $C_{\mathrm{H}}(X)$ exist.
\end{proposition}
\begin{proof}
    Consider $|X_t - X_s|$ for $0 < s < t < T$:
    \begin{align*}
        |X_t - X_s| &\le \kls{\int_s^t b(r, X_r) \mathrm{d}r} + |a_1| |A(B_t^{H_1} - B_s^{H_1})| + |a_2| |A(B_t^{H_2} - B_s^{H_2})| \\
        &\le \int_s^t |b(r,X_r)| \mathrm{d}r + |a_1| |A(B_t^{H_1} - B_s^{H_1})| + |a_2| |A(B_t^{H_2} - B_s^{H_2})| \\
        &\le L|t-s| + |a_1| |A(B_t^{H_1} - B_s^{H_1})| + |a_2| |A(B_t^{H_2} - B_s^{H_2})| \,.
    \end{align*}
    Since the fBM admits a version which is Hölder continuous, we can conclude that $X$ also admits a Hölder continuous version of order $\gamma < \min(H_1,H_2) = H_1$. Since all moments of the Hölder constant of the fBM exist \cite{rascanu2002differential}, all moments of the Hölder constant of $X$ exist.
\end{proof}
\subsection{The conditionally Gaussian process}
To establish the existence of a density for the solution of the mixed SDE, the key ingredient is a conditionally Gaussian process (CGP) that has also been considered by Fan and Zhang in \cite{fan2021moment} as an auxiliary process and was also subject of discussion in~\cite{buthenhoff,sonmez2020mixed}. 
\begin{definition}[Conditionally Gaussian process]
    Let $H_1,H_2 \in (0,1)$ and fix $a_1, a_2 \in \mathbb{R} \setminus \klg{0}$, $A \in 
    \mathbb{R}^{d\times d}$. Let $\varepsilon \in (0,t)$. Define the conditionally Gaussian process (CGP) as
    \begin{align}
        Y(\varepsilon) &\coloneqq X_{t-\varepsilon} + A \kle{a_1 \klr{B_t^{H_1} - B_{t-\varepsilon}^{H_1}} + a_2 \klr{B_t^{H_2} - B_{t-\varepsilon}^{H_2}} }\label{eq:definition-Y} \\
        &= X_t - \int_{t-\varepsilon}^t b(s,X_s) \mathrm{d}s\,.
        \label{remark:expression-CGP}
    \end{align}
\end{definition}
%
    %
%
For the sake of completeness, we recall the following lemmata which represent special cases of~\cite[Section 2]{buthenhoff} and justify the notion \enquote{conditionally Gaussian process}.
\begin{lemma}
\label{lemma:X-Y-expectation}
    If $b$ is bounded, we have 
    \begin{align*}
        \mathbb{E}\kle{|X_t - Y(\varepsilon)|} \le \norm{b}_\infty \varepsilon \,.
    \end{align*}
\end{lemma}
\begin{proof}
    This is a direct consequence of equation \eqref{remark:expression-CGP}.
\end{proof}
\begin{lemma}
\label{lemma:Y-conditionally-gaussian}
    Define 
    \begin{align}
        \xi = X_{t-\varepsilon} + A \int_0^{t-\varepsilon} \klr{ a_1 \klr{K_{H_1}(t,s) - K_{H_1}(t-\varepsilon,s)} + a_2 \klr{K_{H_2}(t,s) - K_{H_2}(t-\varepsilon,s)}}\mathrm{d}W_s \,.\label{eq:def-xi}
    \end{align}
    Then, for all $u \in \mathbb{R}^d$ we have
    \begin{align*}
        \mathbb{E}[\exp\klr{i\kla{u, Y(\varepsilon)}} | \mathcal{A}_{t-\varepsilon}] = \exp\klr{i\kla{u,\xi} - \frac{|A^Tu|^2}{2} \int_{t-\varepsilon}^t \klr{a_1 K_{H_1}(t,s) + a_2 K_{H_2}(t,s)}^2 \mathrm{d}s}\,,
    \end{align*}
    i.e.\ given $\mathcal{A}_{t-\varepsilon}$ the CGP $Y(\varepsilon)$ is conditionally Gaussian with mean $\xi$ and covariance matrix 
    \begin{align}
        \eta^2 = AA^T\int_{t-\varepsilon}^t \klr{a_1 K_{H_1}(t,s) + a_2 K_{H_2}(t,s)}^2 \mathrm{d}s \,.
        \label{eq:variance-conditionally-gaussian-y}
    \end{align}
\end{lemma}
\begin{proof}
    Since $\int_{t-\varepsilon}^t K_{H_{1,2}}(t,s) \mathrm{d}W_s$ is independent of $\mathcal{A}_{t-\varepsilon}$ we have
    \begin{align*}
        &\mathbb{E}\kle{\exp\klr{i\kla{u,A \int_{t-\varepsilon}^t \klr{a_1 K_{H_1}(t,s) + a_2 K_{H_2}(t,s)} \mathrm{d}W_s}} | \mathcal{A}_{t-\varepsilon}} \\
        &\quad = \exp\klr{-\frac{|A^Tu|^2}{2} \int_{t-\varepsilon}^t \klr{a_1 K_{H_1}(t,s) + a_2 K_{H_2}(t,s)}^2 \mathrm{d}s} \,.
    \end{align*}
    Now, due to the integral representation of the fBM we have
    \begin{align*}
        B_t^H - B_{t-\varepsilon}^H = \int_0^{t-\varepsilon} \klr{K_H(t,s) - K_H(t-\varepsilon,s)} \mathrm{d}W_s + \int_{t-\varepsilon}^t K_H(t,s) \mathrm{d}W_s \,.
    \end{align*}
    Therefore we can express $Y(\varepsilon)$ as
    \begin{align*}
        Y(\varepsilon) = \xi + A \kle{a_1 \int_{t-\varepsilon}^t K_{H_1}(T,s) \mathrm{d}W_s + a_2 \int_{t-\varepsilon}^t K_{H_2}(t,s) \mathrm{d}W_s} \,,
    \end{align*}
    where $\xi$ is defined in equation \eqref{eq:def-xi} and measurable with respect to $\mathcal{A}_{t-\varepsilon}$. Therefore we end up with
    \begin{align*}
        &\mathbb{E}[\exp\klr{i\kla{u,Y(\varepsilon)}} | \mathcal{A}_{t-\varepsilon}] \\
        &\quad = \exp\klr{i\kla{u,\xi}} \mathbb{E}\kle{\exp\klr{i\kla{u,A \int_{t-\varepsilon}^t \klr{a_1 K_{H_1}(t,s) + a_2 K_{H_2}(t,s)} \mathrm{d}W_s}} | \mathcal{A}_{t-\varepsilon}} \\
        &\quad = \exp\klr{i\kla{u,\xi} - \frac{|A^Tu|^2}{2} \int_{t-\varepsilon}^t \klr{a_1 K_{H_1}(t,s) + a_2 K_{H_2}(t,s)}^2 \mathrm{d}s} \,.
    \end{align*}
\end{proof}
By extending the previous proof, one can show the following remark:
\begin{remark}\label{remark:YY-cond-gaussian}
    Fix $a_1,a_2,a_1',a_2' \in \mathbb{R}\setminus\klg{0}$ and $A,A' \in \mathbb{R}^{d\times d}$ invertible. Let $X'$ satisfy the SDE \eqref{eq:sde} with diffusion coefficients $a_1'A'$ and $a_2'A'$, i.e.,
    \begin{align*}
        X_t' = x_0 + \int_0^t b(s,X_s') \mathrm{d}s + A' \klr{a_1' B_t^{H_1} + a_2' B_t^{H_2}}\,, \quad t \in [0,T]\,.
    \end{align*}
    Define $Y'(\varepsilon)$, $\xi'$ and $(\eta')^2$ analogously to equations \eqref{eq:definition-Y}, \eqref{eq:def-xi} and \eqref{eq:variance-conditionally-gaussian-y}. Furthermore, let
    \begin{align*}
        \lambda \coloneqq A(A')^T\int_{t-\varepsilon}^t \klr{a_1 a_1' \klr{K_{H_1}(t,s)}^2 + \klr{a_1'a_2 + a_1a_2'} K_{H_1}(t,s) K_{H_2}(t,s) + a_2 a_2' \klr{K_{H_2}(t,s)}^2} \mathrm{d}s \,.
    \end{align*}
    Then, given $\mathcal{A}_{t-\varepsilon}$, the random vector $(Y_t(\varepsilon),Y_t'(\varepsilon))$, ${t \in [0,T]}$, is conditionally Gaussian with mean $(\xi,\xi')$ and covariance block matrix
    \begin{align*}
        \Sigma = \matr{\eta^2 & \lambda\\ \lambda & (\eta')^2} \,.
    \end{align*}
\end{remark}
\subsection{Existence of density}
One consequence of the properties of the CGP is the existence of a density with respect to the Lebesgue measure. A key ingredient is the following lemma, a proof is given in \cite{fournier2010absolute}.
\begin{lemma}
\label{lemma:mu-density-existence}
    Let $\rho \colon \mathbb{R}^d \to [0,\infty)$ be a continuous function and $\delta \in (0,\infty)$. We write $D_\delta = \klg{x \in \mathbb{R}^d; \rho(x) \le \delta}$ and define a function $h_\delta \colon \mathbb{R}^d \to [0,\delta]$ with
    \begin{align*}
        h_\delta(x) = (\inf\klg{|x-z|; z \in D_\delta}) \land \delta\,, \quad x \in \mathbb{R}^d \,,
    \end{align*}
    where we use the convention that $\inf\klg{|x-z|; z \in D_\delta} = 0$ if $D_\delta = \emptyset$. Then:
    \begin{enumerate}
        \item[(a)] $h_\delta$ is vanishing on $D_\delta$, positive on $\mathbb{R}^d \setminus D_\delta$ and globally Lipschitz continuous with Lipschitz constant $1$.
        \item[(b)] For a probability measure $\mu$ on $\mathbb{R}^d$, if for some $\delta > 0$ the measure $\mu_\delta$ given by $\mathrm{d}\mu_\delta/\mathrm{d}\mu = h_\delta$ admits a density then $\mu$ has a density on the set $\klg{x \in \mathbb{R}^d: \rho(x) > 0}$. 
    \end{enumerate}
\end{lemma}
For any $h \in \mathbb{R}^d$ the $m$th order difference operator $\Delta_h^m$ is defined as
\begin{align*}
    \Delta_h f(x) = f(x+h) - f(x)\,, \qquad \Delta_h^m f(x) = \Delta_h\klr{\Delta_h^{m-1}f}(x) \,.
\end{align*}
For $m \in \nat$ and $0 < \gamma < m$, let $\mathcal{C}_b^\gamma(\mathbb{R}^d)$ be the Zygmund space of order $\gamma$ defined as the closure of bounded smooth functions with respect to
\begin{align*}
    \norm{f}_{\mathcal{C}^\gamma_b} \coloneqq \norm{f}_\infty + \sup_{|h|\le 1} \frac{\norm{\Delta^m_h f}_\infty}{|h|^\gamma} \,.
\end{align*}
In the following, we distinguish between the Hölder space of order $\gamma$, $C^\gamma_b(\mathbb{R}^d)$, and the Zygmund space of order $\gamma$, $\mathcal{C}^\gamma_b(\mathbb{R}^d)$, via a calligraphic letter. As it has been shown in \cite{triebel1992theory}, note that $\mathcal{C}^\gamma_b(\mathbb{R}^d) \subset C^\gamma_b(\mathbb{R}^d)$. Moreover, we have the following lemma.
\begin{lemma}\label{remark:relation-difference-operator-and-zygmund-norm}
    For any $m \in \nat$, $\gamma \in (0,m)$, $h \in \mathbb{R}^d$ with $|h| < 1$, $x \in \mathbb{R}^d$ and $\phi \in \mathcal{C}_b^\gamma(\mathbb{R}^d)$ we have
    \begin{align*}
        |\Delta_h^m \phi(x)| \le \norm{\phi}_{\mathcal{C}_b^\gamma} |h|^\gamma \,.
    \end{align*}
\end{lemma}
\begin{proof}
    We start with the definition of the Zygmund norm and multiply it by $|h|^\gamma$, we obtain
    \begin{align*}
        \norm{\phi}_{\mathcal{C}^\gamma_b} |h|^\gamma \ge |h|^\gamma \sup_{|h'| \le 1} \frac{\norm{\Delta_{h'}^m \phi}_\infty}{|h'|^\gamma} \ge \norm{\Delta_h^m \phi}_\infty \ge |\Delta_h^m \phi(x)| \,,
    \end{align*}
    where we have used the definition of the supremum.
\end{proof}
Furthermore, we need the following lemma that has been established in \cite{romito2018simple}.
\begin{lemma}
\label{lemma:mu-density-wrt-lebesgue}
    Let $\mu$ be a finite measure on $\mathbb{R}^d$. Assume that there exist $m \in \nat$, $\gamma \in (0,\infty)$, $s \in (\gamma,m)$ and a constant $K \in (0,\infty)$ such that for all $\phi \in \mathcal{C}_{b}^\gamma(\mathbb{R}^d)$ and $h \in \mathbb{R}^d$ with $|h| \le 1$,
    \begin{align*}
        \kls{\int_{\mathbb{R}^d} \Delta^m_h \phi(x) \mathrm{d}\mu(x)} \le K |h|^s \norm{\phi}_{\mathcal{C}^\gamma_{\mathrm{b}}} \,,
    \end{align*}
    then $\mu$ has a density with respect to the Lebesgue measure on $\mathbb{R}^d$.
\end{lemma}
\begin{theorem}[Existence of density]\label{thm:existence-density}
    If $b$ is bounded, for all $t \in (0,T]$, the law of $X_t$, i.e.\ the solution of equation \eqref{eq:sde}, admits a density with respect to the Lebesgue measure on $\mathbb{R}^d$.
\end{theorem}
\begin{proof}
    The procedure is similar to the proof of Lemma 4.1 in \cite{sonmez2020mixed}. In the following, $M_1, M_2, \hdots$ denote some positive constants. Let $\rho(x) = 1$ and let $h_\delta$ be defined as described in Lemma \ref{lemma:mu-density-existence}. Then we apply Lemma \ref{lemma:mu-density-wrt-lebesgue} to show that the measure $\mu_\delta$ given by $\mathrm{d}\mu_\delta = h_\delta \mathrm{d}\mathbb{P}_{X_t}$ admits a density with respect to the Lebesgue measure, i.e.\ we need to find $m \in \nat$, $\gamma \in (0,\infty)$ and $s \in (\gamma,m)$ such that
    \begin{align*}
        \kls{\mathbb{E}\kle{h_\delta(X_t) \Delta_h^m \phi(X_t)}} \le K |h|^s \norm{\phi}_{\mathcal{C}_b^\gamma}
    \end{align*}
    for all $h$ with $|h|\le 1$, for all $\phi \in \mathcal{C}_b^\gamma$ and some constant $K \in (0,\infty)$. Using the Jenssen inequality, Lemma \ref{lemma:X-Y-expectation}, Lemma \ref{lemma:mu-density-existence} (a), Lemma \ref{remark:relation-difference-operator-and-zygmund-norm} and the Hölder continuity of Zygmund functions, for any $h \in \mathbb{R}^d$ with $|h| < 1$, $m \in \nat$, $\gamma \in (0,m)$ we get
    \begin{align*}
        \kls{\mathbb{E}\kle{h_\delta(X_t) \Delta_h^m \phi(X_t)}} &\le \kls{\mathbb{E}\kle{(h_\delta(X_t) - h_\delta(X_{t-\varepsilon})) \Delta_h^m \phi(X_t)}} \\
        &\quad + \kls{\mathbb{E}\kle{h_\delta(X_{t-\varepsilon})\klr{\Delta_h^m\phi(X_t) - \Delta_h^m \phi(Y(\varepsilon)}}} \\
        &\quad + \kls{\mathbb{E}\kle{h_\delta(X_{t-\varepsilon}) \Delta_h^m\phi(Y(\varepsilon)}} \\
        &\le M_1 \norm{\phi}_{\mathcal{C}^\gamma_b} |h|^\gamma \mathbb{E}\kle{|X_t - X_{t-\varepsilon}|} + M_2 \norm{\phi}_{\mathcal{C}_b^\gamma} \mathbb{E}\kle{|X_t - Y(\varepsilon)|}^\gamma \\
        &\quad + \kls{\mathbb{E}\kle{h_\delta(X_{t-\varepsilon}) \mathbb{E}\kle{\Delta_h^m \phi(Y(\varepsilon)) | \mathcal{A}_{t-\varepsilon}}}} \,.
    \end{align*}
    Now, according to Lemma \ref{lemma:Y-conditionally-gaussian}, $Y(\varepsilon)|\mathcal{A}_{t-\varepsilon} \sim \mathcal{N}(\xi,\eta^2)$ with mean $\xi$ defined in equation \eqref{eq:def-xi} and variance $\eta^2$ given in equation \eqref{eq:variance-conditionally-gaussian-y}.
    So, let $p_y$ be the density of the Gaussian distribution $\mathcal{N}(0,\eta^2)$. Then we have 
    \begin{align*}
        \mathbb{E}\kle{\Delta_h^m \phi(Y(\varepsilon)) |\mathcal{A}_{t-\varepsilon}} &= \int_{\mathbb{R}^d} (\Delta_h^m \phi)(z+\xi) p_y(z) \mathrm{d}z \\
        &= \int_{\mathbb{R}^d} \phi(\xi + z) (\Delta_{-h}^m p_y)(z) \mathrm{d}z \,.
    \end{align*}
    Since according to the mean-value theorem for higher order differences we have
    \begin{align*}
        |(\Delta_{-h}^m p_y)(z)| \le |h|^m \sup_{\xi \in \klg{z-th : t \in [0,m]}} \max_{|\alpha| = m} \kls{D^\alpha p_y(\xi)}\,,
    \end{align*}
    where $\alpha = (\alpha_1,\hdots,\alpha_d)$ is a multi-index with $|\alpha| = \alpha_1 + \hdots + \alpha_d$. We get
    \begin{align*}
        \kls{\mathbb{E}\kle{\Delta_h^m \phi(Y(\varepsilon)) |\mathcal{A}_{t-\varepsilon}}} &\le M_3 \norm{\phi}_\infty |h|^m \le M_3 \norm{\phi}_{\mathcal{C}_b^\gamma} |h|^m \,.
    \end{align*}
    Using Proposition \ref{thm:holder-continuity-XN} and Lemma \ref{lemma:X-Y-expectation} we end up with
    \begin{align*}
        \kls{\mathbb{E}\kle{h_\delta(X_t) \Delta_h^m \phi(X_t)}} &\le M_4 \norm{\phi}_{\mathcal{C}_b^\gamma} \klr{|h|^\gamma \varepsilon^{\beta} + \varepsilon^\gamma + |h|^m} 
    \end{align*}
    where $\beta \in (0,H_1)$. In order to have the conditions of Lemma \ref{lemma:mu-density-wrt-lebesgue} fulfilled, we choose $\gamma = 1$, $\varepsilon = |h|^{2/\beta}$ and $m=4$, i.e.\
    \begin{align*}
        \kls{\mathbb{E}\kle{h_\delta(X_t) \Delta_h^4 \phi(X_t)}} &\le M_5 \norm{\phi}_{\mathcal{C}_b^1} |h|^2
    \end{align*}
    where $s = 2 \in (1,4)$ as required.
    %
    %
    %
    %
    %
\end{proof}
%

\section{Gaussian-type bounds for density}\label{sec:gaussian-bounds}
Now we are in position to prove the existence of Gaussian-type lower and upper bounds for the density. 
\subsection{Representation of density with auxiliary SDE} First of all, let us find an expression for the density of the solution for the SDE \eqref{eq:sde} and adapt \cite[Proposition 3.4]{li2023non} to our setting.
\begin{proposition}
\label{prop:gaussian-like-formula}
    For $t\in [0,T]$ consider
    \begin{align}
        b_t \coloneqq b(t, x_0 + A(a_1B_t^{H_1} + a_2B_t^{H_2})) \,,
        \label{eq:definition-function-b-density-xtn}
    \end{align}
    where the function $b \colon [0,T] \times \mathbb{R}^d \to \mathbb{R}^d$ is the drift function in SDE \eqref{eq:sde}. Under Assumption \ref{assumpt:main-thm}, for each $T_0 > 0$ the density of the solution $X_{T_0}$ of the SDE \eqref{eq:sde} admits the expression
    \begin{align}
        p_{T_0}(x) = \frac{1}{\klr{(2\pi)^d \det(\Sigma)}^{1/2}} \exp\klr{-\frac{1}{2}(x-x_0)^T\Sigma^{-1}(x-x_0)} \Psi_{T_0}(x) \,,
        \label{eq:expression-density-xtn}
    \end{align}
    where $\Sigma = A \sigma^2 A^T$ with
    \begin{align}
        \sigma^2 \coloneqq \int_0^{T_0} \klr{ a_1 K_{H_1}({T_0},t) + a_2 K_{H_2}({T_0},t) }^2 \mathrm{d}t
        \label{eq:sigma-definition}
    \end{align}
    and
    \begin{align}
        \Psi_{T_0}(x) &\coloneqq \mathbb{E}_{\mathbb{P}}\kle{\exp\klr{\int_0^{T_0} \kla{\psi_t, \mathrm{d}W_t} - \frac{1}{2} \int_0^{T_0} |\psi_t|^2 \mathrm{d}t} | A(a_1 B_{T_0}^{H_1} + a_2 B_{T_0}^{H_2}) = x }\,,
    \end{align}
    where $\psi_t$ is defined as in the proof of Proposition \ref{prop:helper-weak-existence} for $h(t) = b_t$.
\end{proposition}
\begin{proof}
    According to Theorem \ref{thm:existence-density} we know that there exists a density $p_{T_0}(x)$ of $X_{T_0}$ with respect to the Lebesgue measure on $\mathbb{R}^d$. Therefore, by definition of the expected value we get for any bounded Borel measurable function $f \colon \mathbb{R}^d \to \mathbb{R}$ 
    \begin{align}
        \int_{\mathbb{R}^d} f(x) p_{T_0}(x) \mathrm{d}x = \mathbb{E}_{\mathbb{P}}\kle{f(X_{T_0})} \,.
        \label{eq:exp-xtn-density}
    \end{align}
    In the proof of Theorem \ref{thm:uniqueness-law} we have already shown that according to Girsanov theorem there exists a probability measure $\mathbb{Q}$ s.t.\ the process
    \begin{align*}
        \Tilde{W}_t = W_t + \int_0^t \psi_s \mathrm{d}s\,, \quad t \in (0,T]\,,
    \end{align*}
    is an $(\mathcal{A}_t)_{t\in [0,T]}$-Brownian motion with respect to the measure $\mathbb{Q}$. Therefore, we can express the right-hand side of equation \eqref{eq:exp-xtn-density} with respect to the new measure $\mathbb{Q}$ as
    \begin{align*}
        \int_{\mathbb{R}^d} f(x) p_{T_0}(x) \mathrm{d}x &= \mathbb{E}_{\mathbb{Q}}\kle{f(X_{T_0}) \exp\klr{\int_0^{T_0} \kla{\psi_t, \mathrm{d}W_t} + \frac{1}{2} \int_0^{T_0} |\psi_t|^2 \mathrm{d}t}} \\
        &= \mathbb{E}_{\mathbb{Q}}\kle{f(X_{T_0}) \exp\klr{\int_0^{T_0} \kla{\psi_t, \mathrm{d}\Tilde{W}_t} - \frac{1}{2} \int_0^{T_0} |\psi_t|^2 \mathrm{d}t}}\,.
    \end{align*}
    Since $(\Tilde{W}_t)_{t \in [0,T]}$ is a $d$-dimensional Brownian motion with respect to the measure $\mathbb{Q}$, the process
    \begin{align*}
        x_0 + A(a_1 \Tilde{B}_{T_0}^{H_1} + a_2 \Tilde{B}_{T_0}^{H_2}) = x_0 + A \int_0^{T_0} \klr{ a_1 K_{H_1}({T_0},t) + a_2 K_{H_2}({T_0},t) } \mathrm{d}\Tilde{W}_t 
    \end{align*}
    has distribution $\mathcal{N}(x_0,\Sigma)$ with the covariance matrix $\Sigma = A\sigma^2A^T$ where
    \begin{align*}
        \sigma^2 = \int_0^{T_0} \klr{ a_1 K_{H_1}({T_0},t) + a_2 K_{H_2}({T_0},t) }^2 \mathrm{d}t \,.
    \end{align*}
    In the following, let us denote its density by 
    \begin{align*}
        \varphi_{T_0}(x) \coloneqq \frac{1}{\klr{(2\pi)^d \det(\Sigma)}^{1/2}} \exp\klr{-\frac{1}{2}(x-x_0)^T\Sigma^{-1}(x-x_0)}  \,,
    \end{align*}
    then we can apply the definition of conditional expectations and end up with
    \begin{align*}
        &\int_{\mathbb{R}^d} f(x) p_{T_0}(x) \mathrm{d}x \\
        &\quad = \int_{\mathbb{R}^d} f(x) \varphi_{T_0}(x) \mathbb{E}_{\mathbb{Q}}\kle{\exp\klr{\int_0^{T_0} \kla{\psi_t, \mathrm{d}\Tilde{W}_t} - \frac{1}{2} \int_0^{T_0} |\psi_t|^2 \mathrm{d}t} | A(a_1 \Tilde{B}_{T_0}^{H_1} + a_2 \Tilde{B}_{T_0}^{H_2}) = x } \mathrm{d}x \,,
    \end{align*}
    i.e., indeed, the density $p_{T_0}(x)$ of $X_{T_0}$ can be expressed by equation \eqref{eq:expression-density-xtn}.
\end{proof}
Note that $A(a_1 \Tilde{B}_{T_0}^{H_1} + a_2 \Tilde{B}_{T_0}^{H_2}) = x$ is equivalent to $(a_1 \Tilde{B}_{T_0}^{H_1} + a_2 \Tilde{B}_{T_0}^{H_2}) = A^{-1}x$, since the matrix $A$ is invertible. From now on until the proof of the main result, by abuse of notation we write $x$ instead of $A^{-1}x$.
The latter proposition gives rise to the following auxiliary SDE which is similar to \cite[Lemma 4.1]{li2023non}.
\begin{lemma}
    \label{lemma:sde-yt}
    Define $f_{H_1,H_2}(T_0,t) \coloneqq a_1 K_{H_1}(T_0,t) + a_2 K_{H_2}(T_0,t)$. Let $x \in \mathbb{R}^d$ be arbitrary. Then, any weak solution of the path-dependent SDE
    \begin{align}
        \mathrm{d}Y_t^x = f_{H_1,H_2}(T_0,t) \kle{\int_t^{T_0} \klr{f_{H_1,H_2}(T_0,u)}^2 \mathrm{d}u}^{-1} \klr{x - \int_0^t f_{H_1,H_2}(T_0,s) \mathrm{d}Y_s^x} \mathrm{d}t + \mathrm{d}W_t
        \label{eq:path-dependent-sde-gaussian-bounds}
    \end{align}
    satisfies $(Y_t^x)_{t \in [0,T_0]} \overset{d}{=} (W_t)_{t \in [0,T_0]} | (a_1 B_{T_0}^{H_1} + a_2 B_{T_0}^{H_2}) = x$.
\end{lemma}
\begin{proof}
    First, let us define
    \begin{align*}
        U_t \coloneqq \int_0^t f_{H_1,H_2}(T_0,s) \mathrm{d}W_s \,, \quad t \in [0,T_0] \,.
    \end{align*}
    Then, $U$ is a bounded $L^2$ martingale with 
    \begin{align*}
        \lim_{t \uparrow T_0} U_t = (a_1 B_{T_0}^{H_1} + a_2 B_{T_0}^{H_2})\,,
    \end{align*}
    and quadratic variation
    \begin{align*}
        \rho(t) = \kla{U}_t = \int_0^t \klr{f_{H_1,H_2}(T_0,s)}^2 \mathrm{d}s \,.
    \end{align*}
    Now, define the process $\Tilde{W}_t \coloneqq U_{\rho^{-1}(t)}$ whose quadratic variation is given by $\kla{\Tilde{W}_t} = \rho(\rho^{-1}(t)) = t$, i.e.\ it represents a Brownian motion. Define
    \begin{align*}
        \phi(U)_t \coloneqq \int_0^t \klr{f_{H_1,H_2}(T_0,s)}^{-1} \mathrm{d}U_s\,,
    \end{align*}
    then it holds that $W_t = \phi(U)_t$. Therefore, using the definition of $\phi$, we can write
    \begin{align*}
        (W,U) = (\phi(\Tilde{W} \circ \rho), \Tilde{W} \circ \rho)
    \end{align*}
    and we can conclude that the conditional distribution of $(W_t)_{t \in [0,T_0]}$ given $U_{T_0} = x$ equals $\phi(Z^x \circ \rho)$, where $(Z_t^x)_{t \in [0,T_0]}$ is a Brownian bridge from $0$ to $x$ which satisfies the SDE
    \begin{align*}
        \mathrm{d}Z_t^x = \frac{x - Z_t^x}{\rho(T_0) - t} \mathrm{d}t + \mathrm{d}\Tilde{W}_t \,, \quad t \in [0,\rho(T_0)) \,.
    \end{align*}
    Now observe that
    \begin{align}
        \mathrm{d}Z^x_{\rho(t)} = \frac{x - Z_{\rho(t)}^x}{\rho(T_0) - \rho(t)} \mathrm{d}\rho(t) + f_{H_1,H_2}(T_0,t) \mathrm{d}W_t\,,
    \end{align}
    i.e.\ $Y^x \coloneqq \phi(Z^x \circ \rho)$ satisfies the SDE
    \begin{align*}
        Y_t^x = \int_0^t \kle{f_{H_1,H_2}(T_0,s)}^{-1} \frac{x - Z^x_{\rho(s)}}{\rho(T_0) - \rho(s)} \mathrm{d}\rho(s) + W_t \,.
    \end{align*}
    Since 
    \begin{align*}
        Z^x_{\rho(t)} = \phi^{-1}(Y^x)_t = \int_0^t f_{H_1,H_2}(T_0,s) \mathrm{d}Y_s^x ,
    \end{align*}
    we end up with
    \begin{align*}
        \mathrm{d}Y_t^x = f_{H_1,H_2}(T_0,t) \kle{\int_t^{T_0} \klr{f_{H_1,H_2}(T_0,u)}^2 \mathrm{d}u}^{-1} \klr{x - \int_0^t f_{H_1,H_2}(T_0,s) \mathrm{d}Y_s^x} \mathrm{d}t + \mathrm{d}W_t \,.
    \end{align*}
\end{proof}
Similarily to \cite[Corollary 4.2]{li2023non}, we can find the solution of the auxiliary SDE by direct computation. 
\begin{corollary}
    \label{cor:expression-yt}
    Set $f \equiv f_{H_1,H_2}$. The solution of SDE \eqref{eq:path-dependent-sde-gaussian-bounds} is given by
    \begin{align}
        \mathrm{d}Y_t^x = \kle{\frac{f(T_0,t)}{\int_0^T f^2(T_0,v) \mathrm{d}v} x - f(T_0,t) \int_0^t \frac{f(T_0,u)}{\int_u^{T_0} f^2(T_0,v) \mathrm{d}v} \mathrm{d}W_u} \mathrm{d}t + \mathrm{d}W_t \,. \label{eq:ansatz-solution-path-dependent-sde}
    \end{align}
\end{corollary}
\begin{proof}
    We perform a direct computation. We insert the ansatz \eqref{eq:ansatz-solution-path-dependent-sde} into the SDE \eqref{eq:path-dependent-sde-gaussian-bounds}, i.e.\ we need to calculate
    \begin{align*}
        \int_0^t f(T_0,s) \mathrm{d}Y^x_s = I_1 - I_2 + I_3\,,
    \end{align*}
    where $I_1$ is given by
    \begin{align*}
        I_1 = \int_0^t f^2(T_0,s) \mathrm{d}s\frac{x}{\int_0^{T_0} f^2(T_0,v) \mathrm{d}v} = x - \int_{t}^{T_0} f^2(T_0,s) \mathrm{d}s \frac{x}{\int_0^{T_0} f^2(T_0,v) \mathrm{d}v}\,,
    \end{align*}
    using the stochastic Fubini, $I_2$ can be calculated as
    \begin{align*}
        I_2 = \int_0^t f(T_0,u) \mathrm{d}W_u - \int_0^t  \frac{f(T_0,u)}{\int_u^{T_0} f^2(T_0,v) \mathrm{d}v} \mathrm{d}W_u \int_t^{T_0} f^2(T_0,s) \mathrm{d}s\,,
    \end{align*}
    and $I_3$ is given by
    \begin{align*}
        I_3 = \int_0^t f(T_0,s) \mathrm{d}W_s \,,
    \end{align*}
    i.e.\ we obtain
    \begin{align*}
        &\int_0^t f(T_0,s) \mathrm{d}Y_s^x \\
        &\quad = x - \int_t^{T_0} f^2(T_0,s) \mathrm{d}s \frac{x}{\int_0^{T_0} f^2(T_0,v) \mathrm{d}v} + \int_t^{T_0} f^2(T_0,s) \mathrm{d}s \int_0^t \frac{f(T_0,u)}{\int_u^{T_0} f^2(T_0,v) \mathrm{d}v} \mathrm{d}W_u \,.
    \end{align*}
    Therefore, we can conclude that the solution is given by \eqref{eq:ansatz-solution-path-dependent-sde}.
\end{proof}

As already mentioned, we will utilize exponential Orlicz spaces as a tool in proving our main result. To provide the necessary background, we now recall some fundamental definitions and properties of these spaces.

\subsection{Exponential Orlicz space} A useful generalization of $L^p$ spaces are Orlicz spaces endowed with the Luxemburg norm \cite{buldygin2000metric}. Let $(\Omega,\mathcal{A},\mathbb{P})$ be some probability space. Then, the Orlicz space $L_U(\Omega)$ generated by the function $U(x)$ is the space of random variables $\xi$ on $(\Omega,\mathcal{A})$ such that for some constant $K_\xi > 0$ we have
\begin{align*}
    \mathbb{E}\kle{U\klr{\frac{\xi}{K_\xi}}} < \infty \,.
\end{align*}
The Orlicz space $L_U(\Omega)$ is the Banach space with respect to the Luxemburg norm 
\begin{align*}
    \norm{\xi}_U = \inf\klg{r > 0; \mathbb{E}\kle{U\klr{\frac{\xi}{r}}} \le 1} \,.
\end{align*}
Hence, for any $\lambda \in \mathbb{R}$ and $\xi, \xi_1,\xi_2 \in L_U(\Omega)$ we have
\begin{enumerate}
    \item[(i)] $\lambda \xi \in L_U(\Omega)$ with $\norm{\lambda \xi}_U = |\lambda| \norm{\xi}_U$,
    \item[(ii)] $\xi_1 + \xi_2 \in L_U(\Omega)$ with $\norm{\xi_1 + \xi_2}_U \le \norm{\xi_1}_U + \norm{\xi_2}_U$.
\end{enumerate}
In the following, the Orlicz space generated by $U(x) = \exp\klr{x^2} - 1$ will be denoted by $L_E(\Omega)$ and its Luxemburg norm by $\norm{\cdot}_E$. We need the following lemma that can be found for example in \cite{buldygin2000metric}:
\begin{lemma}\label{lemma:orlicz-exp-bound}
    If $\xi \in L_E(\Omega)$, then for any $\lambda \in \mathbb{R}$ we have
    \begin{align*}
        \mathbb{E}\kle{\exp\klr{\lambda |\xi|}} \le 2 \exp\klr{\frac{\lambda^2 \norm{\xi}_E^2}{4}} \,.
    \end{align*}
\end{lemma}
Now we can take a look at the Brownian dependent term of the auxiliary process \eqref{eq:ansatz-solution-path-dependent-sde} defined in the previous subsection.
\begin{lemma}\label{lemma:gt-orlicz}
    Define the process $(G_t)_{t \in [0,T_0]}$ as the Brownian dependent term of the stochastic process given in equation \eqref{eq:ansatz-solution-path-dependent-sde}, i.e.\
    \begin{align*}
        G_t \coloneqq f(T_0,t) \int_0^t \frac{f(T_0,u)}{\int_u^{T_0} f^2(T_0,v) \mathrm{d}v} \mathrm{d}W_u \,.
    \end{align*}
    Then, $|G_t| \in L_E(\Omega)$ for all $t \in [0,T_0)$ where its Luxemburg norm satisfies $\norm{|G_t|}_E = (8/3)^{1/2} \kappa_t$ where
    \begin{align*}
        \kappa_t^2 = \int_0^t f^2(T_0,t) \frac{f^2(T_0,u)}{\klr{\int_u^{T_0} f^2(T_0,v) \mathrm{d}v}^2} \mathrm{d}u \,.
    \end{align*}
\end{lemma}

\begin{proof}
    Let $K^2 = 2\kappa_t^2 ad$ where $t \in [0,T_0)$ is fixed and $a > 1$ is arbitrary. Since $G_t \sim \mathcal{N}(0,\mathbbm{1}_{d\times d}\kappa_t^2)$ where $\mathbbm{1}_{d\times d}$ represents the ($d\times d$)-identity matrix, we have
    \begin{align*}
        \mathbb{E}\kle{\exp\klr{\frac{|G_t|^2}{K^2}}} &= \klr{\klr{1 - \frac{1}{a}}^{-d/2}}^{1/d} = \klr{1 - \frac{1}{a}}^{-1/2} < \infty \,,
    \end{align*}
    i.e.\ by definition we have $|G_t| \in L_E(\Omega)$. To estimate an upper bound for the norm, we note that 
    \begin{align*}
        \inf \klg{a > 1: \klr{1 - \frac{1}{a}}^{-1/2} \le 2} = \frac{4}{3} \,,
    \end{align*}
    hence, $\norm{|G_t|}_E = (8/3)^{1/2} \kappa_t$.
\end{proof}

From now on set $\kappa_t$ for $t\in [0,T_0]$ as in Lemma \ref{lemma:gt-orlicz}. We proceed with the following result.

\begin{lemma}\label{lemma:integral-sigmat-existence}
    Let $H \coloneqq \min(H_1,H_2)$ and $H' \coloneqq \max(H_1,H_2)$. There exists a constant $C(H_1,H_2) > 0$ such that
    \begin{align*}
        \int_0^T t^{1/2-H} \kappa_t \mathrm{d}t \le  C(H_1,H_2) \klgcases{
        |a/a'| T_0^{2 - 2H - H'} + T_0^{2-H-2H'} & \text{for }
            H,H' \in (0,1/2) ,\\
        |a/a'| T_0^{1 - H'} + T_0^{3/2-H-H'} & \text{for }
            H \in (0,1/2], H' \in [1/2,1) ,\\
        (|a/a'| + 1) T_0^{3/2-H'-H} & \text{for } H,H' \in (1/2,1) .
        } 
    \end{align*}
    Here we set $a = a_1$ and $a' = a_2$ if $H_1 \le H_2$ and $a = a_2$ and $a' = a_1$ if $H_1 > H_2$. 
\end{lemma}
\begin{proof}
    First, let us find a simplified expression for $\kappa_t^2$. By definition, it is given by
    \begin{align*}
        \kappa_t^2 = \int_0^t f^2(T_0,t) \frac{f^2(T_0,u)}{\klr{\int_u^{T_0} f^2(T_0,v) \mathrm{d}v}^2} \mathrm{d}u \,.
    \end{align*}
    Substitute $z = \int_u^T f^2(T_0,v) \mathrm{d}v$. Then according to the Leibniz integral rule we get
    \begin{align*}
        \kappa_t^2 = -\int_{z(0)}^{z(t)} f^2(T_0,t) \frac{1}{z^2} \mathrm{d}z
        = f^2(T_0,t) \klr{\frac{1}{z(t)} - \frac{1}{z(0)}} \le \frac{f^2(T_0,t)}{\int_t^{T_0} f^2(T_0,v) \mathrm{d}v} \,.
    \end{align*}
    Since
    \begin{align*}
        f^2(T_0,t) = \klr{a_1 K_{H_1}(T_0,t) + a_2 K_{H_2}(T_0,t)}^2 \ge a_2^2 K_{H_2}^2(T_0,t)\,,
    \end{align*}
    we can further estimate $\kappa_t^2$ by
    \begin{align*}
        \kappa_t^2 \le \klr{a_1/a_2 K_{H_1}(T_0,t) + K_{H_2}(T_0,t)}^2 \klr{\int_t^{T_0} K_{H_2}^2(T_0,v) \mathrm{d}v}^{-1} \,.
    \end{align*}
    Hence, we can estimate the above integral as
    \begin{align*}
        \int_0^{T_0} t^{1/2-H} \kappa_t \mathrm{d}t &\le \int_0^{T_0} t^{1/2-H} \klr{|a_1/a_2| |K_{H_1}({T_0},t)| + |K_{H_2}({T_0},t)|} \klr{\int_t^{T_0} K_{H_2}^2({T_0},v) \mathrm{d}v}^{-1/2} \mathrm{d}t \\
        &\eqqcolon |a_1/a_2| I_1 + I_2 \,.
    \end{align*}
    Note that if $H \le 1/2$, we can make use of $t^{1/2-H}\le T^{1/2-H}$. Now we apply the bounds for the covariance kernel given in equations \eqref{eq:cov-bound-1}, \eqref{eq:cov-bound-2} and \eqref{eq:cov-bound-3} and substitute $x{T_0} = t$. \\
    \indent Without loss of generality, assume that $H_1 \le H_2$. In total we distinguish between four cases:
    \\\\\textit{Case 1: $H_1,H_2 \in (0,1/2)$.} Due to
    \begin{align*}
        s^{1-2H_2} ({T_0}-s)^{2H_2-1} \ge t^{1-2H_2} ({T_0}-t)^{2H_2-1}
    \end{align*}
    for all $t \le s \le {T_0}$, we obtain
    \begin{align*}
        \klr{\int_t^{T_0} K_{H_2}^2({T_0},s) \mathrm{d}s}^{-1/2} \le C_3^
        {-1}(H_2) ({T_0}-t)^{-H_2}\,,
    \end{align*}
    where $C_3(H_2)$ results from the estimation \eqref{eq:cov-bound-3}. Hence, we have
    \begin{align*}
        I_1 &\le {T_0}^{2-2H_1-H_2} C'(H_1,H_2) B(3/2-H_1,1-H_2)\,, \\
        I_2 &\le {T_0}^{2-H_1-2H_2} C'(H_1,H_2) B(3/2-H_2,1-H_2) \,,
    \end{align*}
    where $C'(H_1,H_2) > 0$ is due to the bounds \eqref{eq:cov-bound-1}, \eqref{eq:cov-bound-2} and \eqref{eq:cov-bound-3}, and $B(a,b)$ represents the beta function which is defined as
    \begin{align*}
        B(a,b) = \int_0^1 x^{a-1} (1-x)^{b-1} \mathrm{d}x
    \end{align*}
    for $a,b > 0$. 
    \\\\\textit{Case 2: $H_1 \in (0,1/2), H_2 \in (1/2,1)$.}
    In this case we obtain for $I_1$ that
    \begin{align*}
        I_1 
        &\le C'(H_1,H_2) {T_0}^{1-H_2} \int_0^1 x^{H_1 - 1/2} (1-x)^{-H_2} \mathrm{d}x \\
        &= C'(H_1,H_2) {T_0}^{1 - H_2} B(H_1 + 1/2, 1 - H_2)\,,
    \end{align*}
    and for $I_2$ we obtain
    \begin{align*}
        I_2 &\le C'(H_1,H_2) {T_0}^{3/2-H_1-H_2} \int_0^1 x^{-H_2 + 1/2} (1-x)^{-1/2} \mathrm{d}x \\
        &= C'(H_1,H_2) {T_0}^{3/2-H_1-H_2} B(3/2 - H_2, 1/2) \,.
    \end{align*}
    Here, $C'(H_1,H_2)$ is again a constant resulting from the bounds on the covariance kernel. 
    \\\\\textit{Case 3: $H_1,H_2 \in (1/2,1)$.} To estimate this case we perform the same steps as in the previous case for the second integral to obtain
    \begin{align*}
        I_1,I_2 \le C'(H_1,H_2) {T_0}^{3/2-H_1-H_2} \klgcases{
            B(2 - H_1 - H_2, 1/2+H_1-H_2) & \text{for } I_1,\\
            B(2 - 2H_2, 1/2) & \text{for } I_2.
        } 
    \end{align*}
    \\\\\textit{Case 4: $H_1=1/2$ $,H_2 \in (0,1)$.} Observe that $K_{H_1} \equiv 1$, so this case can be estimated using similar arguments as in the previously considered cases.\\
    \indent In total, we end up with the desired result.
\end{proof}
\begin{lemma}\label{lemma:integralt-orlicz}
    Let $H = \min(H_1,H_2)$. Define for ${T_0} > 0$
    \begin{align*}
        I_{T_0} \coloneqq \int_0^{T_0} t^{1/2-H} |G_t| \mathrm{d}t \,.
    \end{align*}
    Then, $I_{T_0} \in L_E(\Omega)$ and the Luxemburg norm satisfies
    \begin{align*}
        \norm{I_{T_0}}_E \le \klr{\frac{8}{3}}^{1/2} \int_0^{T_0} t^{1/2-H} \kappa_t \mathrm{d}t\,.
    \end{align*}
\end{lemma}
\begin{proof}
    Let $\varepsilon > 0$ be arbitrary small. Set $\varepsilon = t_0^{(n)} < t_1^{(n)} < t_2^{(n)} < \hdots < t_n^{(n)} = {T_0}$ such that for any $i = 1,\hdots,n$
    \begin{align*}
        t_i^{(n)} - t_{i-1}^{(n)} = \frac{{T_0-\varepsilon}}{n} \,.
    \end{align*}
    Since $|G_t| \in L_E(\Omega)$ and $L_E(\Omega)$ is a Banach space, we also know that for all $n \in \mathbb{N}$ the sum 
    \begin{align*}
        \xi_n^\varepsilon \coloneqq \sum_{i=1}^n t_i^{1/2-H} |G_{t_i}| (t_i^{(n)} - t_{i-1}^{(n)})
    \end{align*}
    is in $L_E(\Omega)$, whereby its norm satisfies according to Lemma \ref{lemma:gt-orlicz}
    \begin{align}
        \norm{\xi_n^\varepsilon}_E \le \klr{\frac{8}{3}}^{1/2} \int_\varepsilon^{T_0} t^{1/2-H} \kappa_t \mathrm{d}t \eqqcolon R^\varepsilon
        \label{eq:norm-sn-estimation}
    \end{align}
    for all $n \in \nat$. Note that it has been shown in Lemma \ref{lemma:integral-sigmat-existence} that the integral in equation \eqref{eq:norm-sn-estimation} is finite. Let $R_n^\varepsilon \coloneqq \norm{\xi_n^\varepsilon}_E$. Since $R_n^\varepsilon \le R$ for all $n \in \nat$, we have due to the monotonicity of the exponential function and by definition of the Luxemburg norm
    \begin{align*}
        \mathbb{E}\kle{\exp\klr{\frac{(\xi_n^\varepsilon)^2}{R^2}}} \le 2 \,.
    \end{align*}
    Now we want to show that the limit
    \begin{align*}
        \lim_{\varepsilon \to 0}\lim_{n \to \infty} \xi_n^\varepsilon = \lim_{\varepsilon \to 0}\lim_{n \to \infty} \sum_{i=1}^n t_i^{1/2-H} |G_{t_i}|(t_i^{(n)} - t_{i-1}^{(n)}) = \int_0^{T_0} t^{1/2-H} |G_t| \mathrm{d}t = I_{T_0}
    \end{align*}
    is also an element of the Orlicz space $L_E(\Omega)$. To show this, we choose $R > 0$ as defined in equation~\eqref{eq:norm-sn-estimation} and get by using the continuity of the exponential function and Fatou's lemma
    \begin{align*}
        \mathbb{E}\kle{\exp\klr{\frac{I_{T_0}^2}{R^2}}} &\le \liminf_{\varepsilon \to 0} \liminf_{n \to \infty} \mathbb{E}\kle{\exp\klr{\frac{(\xi_n^\varepsilon)^2}{R^2}}} \le 2 < \infty\,.
    \end{align*}
    Therefore, $I_{T_0} \in L_E(\Omega)$ with $\norm{I_{T_0}}_E \le R$.
\end{proof}
\subsection{Proof of main result}\label{sec:main-result-proof}
%
%
\begin{proof}[Proof of Theorem~\ref{thm:main-thm}]
    In the following, $M_1, M_2, \hdots$ denote some positive constants that are independent of $x$. Without loss of generality, we set $x_0 = 0$. According to Theorem \ref{thm:existence-density} we know that there exists a density $p_{T_0}(x)$ of $X_{T_0}$ with respect to the Lebesgue measure on $\mathbb{R}^d$. Furthermore, in Proposition \ref{prop:gaussian-like-formula} we have shown that the density admits the expression
    \begin{align*}
        p_{T_0}(x) = \frac{1}{\klr{(2\pi)^d \det(\Sigma)}^{1/2}} \exp\klr{-\frac{1}{2}x^T\Sigma^{-1}x} \Psi_{T_0}(x)\,,
    \end{align*}
    where $\Psi_{T_0}(x)$ is given by
    \begin{align*}
        \Psi_{T_0}(x) = \mathbb{E}_{\mathbb{P}}\kle{\exp\klr{\int_0^{T_0} \kla{\psi_t, \mathrm{d}W_t} - \frac{1}{2} \int_0^{T_0} |\psi_t|^2 \mathrm{d}t} | A(a_1 B_{T_0}^{H_1} + a_2 B_{T_0}^{H_2}) = x }\,,
    \end{align*}
    and $\Sigma = A\sigma^2A^T$ with $\sigma^2$ defined in equation \eqref{eq:sigma-definition}. In Lemma \ref{lemma:sde-yt} we have shown that
    \begin{align*}
        (Y_t^{A^{-1}x})_{t \in [0,{T_0}]} \overset{d}{=} (W_t)_{t \in [0,{T_0}]} | A(a_1 B_{T_0}^{H_1} + a_2 B_{T_0}^{H_2}) = x\,,
    \end{align*}
    where $(Y_t^{A^{-1}x})_{t \in [0,{T_0}]}$ admits the expression shown in Corollary \ref{cor:expression-yt}. Therefore, we can write $\Psi_{T_0}(x)$ as
    \begin{align*}
        \Psi_{T_0}(x) &= \mathbb{E}\kle{\exp\klr{\int_0^{T_0} \kla{\psi_t, \mathrm{d}Y_t^{A^{-1}x}} - \frac{1}{2} \int_0^{T_0} |\psi_t|^2 \mathrm{d}t}} \\
        &= \mathbb{E}\kle{\exp\klr{\int_0^{T_0} \kla{\psi_t, \mathrm{d}W_t} - \frac{1}{2} \int_0^{T_0} |\psi_t|^2 \mathrm{d}t + \int_0^{T_0} \kla{\psi_t,\klr{\frac{f({T_0},t)}{\int_0^{T_0} f^2({T_0},v) \mathrm{d}v}A^{-1}x - G_t}}\mathrm{d}t}} \,,
    \end{align*}
    where the process $(G_t)_{t \in [0,{T_0}]}$ is defined in Lemma \ref{lemma:gt-orlicz}. \\\\
    \textit{Step 1}: Our goal is to show that $\Psi_{T_0}(x)$ admits an upper bound. Due to Cauchy-Schwarz we have
    \begin{align*}
        \Psi_{T_0}(x) &\le \mathbb{E}^{1/2}\kle{\exp\klr{\int_0^{T_0} 2\kla{\psi_t, \mathrm{d}W_t} -  \int_0^{T_0} |\psi_t|^2 \mathrm{d}t}} \\
        &\quad \cdot \mathbb{E}^{1/2}\kle{\exp\klr{\int_0^{T_0} 2\kla{\psi_t, \klr{\frac{f({T_0},t)}{\int_0^{T_0} f^2({T_0},v) \mathrm{d}v} A^{-1}x - G_t}}\mathrm{d}t}} \,.
    \end{align*}
    Since $(\int_0^{T_0} \kla{\psi_t, \mathrm{d}W_t})_{{T_0} \ge 0}$ is a local martingale, it is well known that its Doléans-Dade exponential satisfies~\cite[Section 4]{li2023non}
    \begin{align*}
        \mathbb{E}\kle{\exp\klr{\int_0^{T_0} \kla{\psi_t, \mathrm{d}W_t}} - \frac{1}{2} \int_0^{T_0} |\psi_t|^2 \mathrm{d}t} \le M_1 {\mathbb{E}^{1/2}\kle{\int_0^{T_0} |\psi_t|^2 \mathrm{d}t}} \,.
    \end{align*}
    In Proposition \ref{prop:helper-weak-existence} it has been shown that $|\psi_t|$ is bounded. Using this proposition we can estimate
    \begin{align*}
        \Psi_{T_0}(x) &\le M_2 \mathbb{E}^{1/2}\kle{\exp\klr{\int_0^{T_0} 2\kla{\psi_t, \klr{\frac{f({T_0},t)}{\int_0^{T_0} f^2({T_0},v) \mathrm{d}v}A^{-1}x - G_t}} \mathrm{d}t}} \\
        &\le M_2 \mathbb{E}^{1/2}\kle{\exp\klr{M_2 \int_0^{T_0}t^{1/2-H}\frac{|f({T_0},t)|}{\int_0^{T_0} f^2({T_0},v) \mathrm{d}v}|A^{-1}x|\mathrm{d}t} \exp\klr{M_2 \int_0^{T_0} t^{1/2-H} |G_t| \mathrm{d}t}} \,.
    \end{align*}
    The first factor in the expectation can be estimated as in Lemma \ref{lemma:integral-sigmat-existence}, in particular, we obtain using the bounds on the covariance kernel that
    \begin{align*}
        &\int_0^{T_0} t^{1/2-H} \frac{|f({T_0},t)|}{\int_0^{T_0} f^2({T_0},v) \mathrm{d}v} \mathrm{d}t \\
        &\quad \le C(H_1,H_2, T_0) \klgcases{
            |a|/|a'|^2\, {T_0}^{2 - 2H - 2H'} + 1/|a'|\, {T_0}^{2 - H - 3H'} & \text{for } H,H' \in (0,1/2) \\
            1/|a|\, {T_0}^{1 - 2H} + |a'|/|a|^2\, {T_0}^{3/2 - 3H} & \text{for } H \in (0,1/2], H' \in [1/2,1) \\
            (|a|/|a'|^2 + 1/|a'|) {T_0}^{3/2 - H - 2H'} & \text{for } H,H' \in (1/2,1)
        }
    \end{align*}
    with $H = \min(H_1,H_2)$, $H' = \max(H_1,H_2)$ and $a,a'$ defined as in Lemma \ref{lemma:integral-sigmat-existence}. In total we get
    \begin{align*}
        \Psi_{T_0}(x) \le M_2  \exp\klr{M_3(T_0) |A^{-1}x|} \mathbb{E}^{1/2}\kle{\exp\klr{M_2 \int_0^{T_0} t^{\frac{1}{2}-H}|G_t| \mathrm{d}t}} \,.
    \end{align*}
    According to Lemma \ref{lemma:integralt-orlicz} the exponent is in the Orlicz space $L_E(\Omega)$, i.e., we can apply Lemma \ref{lemma:orlicz-exp-bound} to conclude that
    \begin{align*}
        \Psi_{T_0}(x) \le M_2 \exp\klr{M_3(T_0) |A^{-1}x| + M_4(T_0)}\,,
    \end{align*}
    where $M_3(T_0),M_4(T_0) > 0$ depend on $T_0$ for $H > 1/2$ only. Hence the density admits the upper bound
    \begin{align*}
        p_{T_0}(x) &\le \frac{M_2}{(2\pi \sigma^2(T_0))^{d/2}\det(A)} \exp\klr{-\frac{1}{2}x^T\Sigma^{-1}(T_0)x + M_3(T_0) |A^{-1}x| + M_4(T_0)} \\
        & \leq \frac{M_2}{(2\pi \sigma^2(T_0))^{d/2}\det(A)} \exp\klr{-\frac{|A^{-1}x|^2}{2\sigma^2(T_0)} + M_3(T_0) |A^{-1}x| + M_4(T_0)}.
    \end{align*}
    Now if $|x| \leq C$ for some given positive constant $C$ we can estimate
    $$p_{T_0}(x) \leq \frac{M_5 (T_0)}{(2\pi \sigma^2(T_0))^{d/2}} \exp\klr{-M_6 \frac{|x|^2}{2\sigma^2(T_0)} },$$
    where $M_5(T_0)>0$ only depends on $T_0$ for $H > 1/2$. Otherwise, if $|x| > C$ we estimate
    \begin{align*}
        p_{T_0}(x) &\le  \frac{M_5(T_0)}{(2\pi \sigma^2(T_0))^{d/2}\det(A)} \exp\klr{-\frac{|A^{-1}x|^2}{2\sigma^2(T_0)} + M_3(T_0) |A^{-1}x| }\\
        &\le \frac{M_5(T_0)}{(2\pi \sigma^2(T_0))^{d/2}\det(A)} \exp\klr{-M_6 \frac{|x|^2}{2\sigma^2(T_0)} \klr{ 1- \frac{M_3(T_0) \sigma^2(T_0)}{|x|} } } \\
        &\leq \frac{M_5(T_0)}{(2\pi \sigma^2(T_0))^{d/2}\det(A)} \exp\klr{-M_7 \frac{|x|^2}{2\sigma^2(T_0)}  } 
    \end{align*}
    using also the fact that $M_3(T_0) \sigma^2(T_0) \leq C$ for some constant independent of $T_0$. Hence, in total we have
    \begin{align*}
        p_{T_0}(x) \le \frac{C_1}{\sigma^d(T_0)} \exp\klr{-\frac{C_2|x|^2}{\sigma^2(T_0)}} \,.
    \end{align*}
    \\\\
    \textit{Step 2}: Now we want to show that $\Psi_{T_0}(x)$ admits a lower bound. Due to Jensen's inequality, we have
    \begin{align*}
        \Psi_{T_0}(x) \!\ge\! M_7 \exp\klr{\mathbb{E}\kle{-\kls{\int_0^{T_0} \!\!\kla{\psi_t, \mathrm{d}W_t}\! - \!\frac{1}{2} \int_0^{T_0} \!\!|\psi_t|^2 \mathrm{d}t\! + \!\int_0^{T_0} \!\!\kla{\psi_t, \klr{\frac{f({T_0},t)}{\int_0^{T_0} f^2({T_0},v) \mathrm{d}v}A^{-1}x - G_t}}\mathrm{d}t}}} \,.
    \end{align*}
    Now we apply the triangle inequality and proceed as in the first step to estimate the remaining term. We end up with
    \begin{align*}
        p_{T_0}(x) &\ge 
        \frac{C_1'}{\sigma^d(T_0)} \exp\klr{-\frac{C_2'|x|^2}{\sigma^2({T_0})}} \,,
    \end{align*}
    where, again, $C_1',C_2'$ denote some $a_1,a_2,H_1,H_2,x_0$ and ${T_0}$-dependent constants. 
\end{proof}
\begin{remark}
    Denote by $K = K(H_1,H_2,A,a_1,a_2) > 0$ some positive constant that depends on the Hurst indices and the diffusion constants. Then, the previous proof dictates the $T_0$ dependence of the constants $C_1,C_1',C_2,C_2'$ to be
    \begin{align*}
        C_1,C_1' \propto  \klgcases{
            1 & \text{for } H \le 1/2\\
            \exp\klr{K ({T_0}^{3-2H-2H'} + {T_0}^{3/2-H-2H'})} & \text{for } H > 1/2 
        } \,, \qquad
        C_2,C_2' \propto 1 \,.
    \end{align*}
    Therefore, for $H \le 1/2$ we see that our results are optimal.
\end{remark}
%
    %
    %
%
%
%
    
%
%
    
%
    
%


\addcontentsline{toc}{section}{References}

\end{document}